\documentclass[12pt]{amsart}
 \usepackage{changes}
\usepackage{lineno,hyperref}
\usepackage{amsmath,amsfonts,amsthm}
\usepackage{amssymb,latexsym}
\usepackage{cases}
\usepackage[numbers]{natbib}
\usepackage{booktabs}
\usepackage{bm}
\usepackage{cleveref}

\usepackage{multirow}
\renewcommand\appendix{\setcounter{secnumdepth}{-2}}

\usepackage{url,cancel}
\usepackage{float}

\usepackage{caption,enumitem}
\newtheorem{thm}{Theorem}[section]

\newtheorem{lem}[thm]{Lemma}
\newtheorem{prop}[thm]{Proposition}
\newtheorem{cor}[thm]{Corollary}

\newtheorem{ex}[thm]{Example}

\theoremstyle{definition}
\newtheorem{defn}[thm]{Definition}

\newcommand{\gen}{\mathrm{gen}}

\newcommand{\ord}{\mathrm{ord}}

\newcommand{\pic}{\mbox{Pic}}
\newcommand{\gal}{\mbox{Gal}}
\newcommand{\rank}{\mathrm{rank}}

\newcommand{\Q}{{\mathbb Q}}

\newcommand{\cO}{{\mathcal O}}

\makeatletter
\@addtoreset{equation}{section}
\makeatother
\numberwithin{equation}{section}

\begin{document}

\author{Zilong He}
\address{}
\email{}
\author{Yong Hu}
\address{ }
\email{ }
\author{Fei Xu}
\address{ }
\email{ }
\title[ ]{On indefinite $k$-universal integral quadratic forms over number fields}
\thanks{ }
\subjclass[2010]{11E12, 11E08, 11E20, 11R11}
\date{\today}
\keywords{integral quadratic forms, local-global principle, integral representation, universal quadratic forms, quadratic fields}
\begin{abstract}
 An integral quadratic lattice is called indefinite $k$-universal if it  represents all integral quadratic lattices of rank $k$ for a given positive integer $k$.

 For $k\geq 3$, we prove that the indefinite $k$-universal property satisfies the local-global principle over number fields.

 For $k=2$, we show that a number field $F$ admits an integral quadratic lattice which is locally $2$-universal but not indefinite 2-universal if and only if the class number of $F$ is even. Moreover, there are only finitely many classes of such lattices over $F$.

 For $k=1$, we prove that $F$ admits a classic integral lattice which is locally classic $1$-universal but not classic indefinite $1$-universal if and only if $F$ has a quadratic unramified extension where all dyadic primes of $F$ split completely. In this case, there are infinitely many classes of such lattices over $F$. All quadratic fields with this property are determined.

\end{abstract}
\maketitle

\section{Introduction}

As an extension of the four squares theorem of Lagrange, Ramanujan  determined in \cite{Ram} all diagonal positive definite quaternary integral quadratic forms which represent all positive integers. Later, Dickson and his school called these forms universal, and  further studied  forms of this kind extensively in \cite{Di-27}, \cite{Di}, \cite{Ro}, \cite{Will} etc. in both the positive definite and the indefinite cases. Instead of classifying all universal quadratic forms, Conway and Schneeberger provided a surprisingly simple criterion, according to which a classic positive quadratic form is universal if and only if it represents all positive integers up to 15. In \cite{Bh}, Bhargava gave a simple proof of the Conway-Schneeberger theorem, and then corrected and completed the previous results in \cite{Will} by using this theorem.

Representation of an integral quadratic form by another is a natural generalization of representation of an integer by an integral quadratic form. Generalizing Lagrange's four squares theorem in this direction, Mordell proved that every binary positive classic quadratic form can be represented by the sum of five squares in \cite{Mor}.  Ko \cite{Ko} further proved that every $k$-ary positive classic quadratic form can be represented by the sum of $k+3$ squares for $3\leq k\leq 5$. A general representation theory for positive definite quadratic forms has been established by Hsia, Kitaoka and Kneser  in \cite{hsia_positivedefinite_1998}. Their result has been improved by Ellenberg and Venkatesh over $\Bbb Z$ in \cite{EV}. A natural extension of universal form to higher dimensions is a positive integral quadratic form representing all $k$-ary positive integral quadratic forms for a given positive integer $k$. B. M. Kim, M.-H. Kim and S. Raghavan called such a quadratic form $k$-universal and classified all 2-universal positive quinary  diagonal   quadratic forms over $\Bbb Z$ in \cite{kim_2universal_1997}. B. M. Kim, M.-H. Kim and B.-K. Oh \cite{kim_2universal_1999} further determined all 2-universal quinary positive classic quadratic forms. They also established an analogue of Conway-Schneeberger's theorem for 2-universal forms over $\Bbb Z$. It is proved in \cite{Oh} that a positive classic quadratic form  over $\Bbb Z$ is $8$-universal if and only if it represents the two special forms $I_8$ and $E_8$. We refer the reader to the survey article \cite{Kim} for the case $k\geq 3$ as well as some related results over other number fields.

A natural question of universal property arises for indefinite integral quadratic forms over number fields. Since the ring of integers of a number field is not necessarily a principal ideal domain in general, one can use integral quadratic lattices in a more general setting. In order to distinguish from the positive definite case, we call an integral quadratic lattice \emph{indefinite $k$-universal} if it represents all integral quadratic lattices of rank $k$ for a given positive integer $k$. Indeed, when $k=1$, such universal property has been studied in \cite{xu_indefinite_2020}. It turns out that indefinite 1-universal property satisfies the local-global principle over $\Bbb Z$ by strong approximation for spin groups. The local universal conditions have been given in \cite{beli_universal_2020}, \cite{EG} and \cite{xu_indefinite_2020}. Therefore this provides an effective algorithm for determining an indefinite 1-universal form over $\Bbb Z$.
However, the local-global principle is not always true over general number fields. For example, Estes and Hsia have shown that the sum of three squares over certain imaginary quadratic fields is locally 1-universal but not indefinite 1-universal in \cite{EH0} and \cite{EH}. In \cite{xu_indefinite_2020}, the third named author and Zhang have proved that
a number field $F$ admits an integral lattice which is locally 1-universal but not indefinite 1-universal if and only if the class number of $F$ is even.
In this paper, we will extend this result to $k\geq 2$ (see Corollary \ref{2.3} and Theorem \ref{5.1}).

\begin{thm} \label{1.1} Let $k$ be a positive integer and $F$ be a number field.

(1) For $k\ge 3$, an integral quadratic lattice over $F$ is indefinite $k$-universal if and only if it is locally $k$-universal.

(2) There exists an integral quadratic lattice which is locally $2$-universal but not indefinite $2$-universal over $F$ if and only if the class number of $F$ is even.  In this case, there are only finitely many classes of such integral quadratic lattices over $F$.
\end{thm}

Saying  a quadratic lattice over $F$ is integral means that the norm of the lattice is contained in the ring of integers of $F$ (see Definition \ref{int}).
One can also define the notion of classic indefinite $k$-universality (see Definition\;\ref{1.3}) by restricting to classic integral lattices (see Definition \ref{int}). For the classic indefinite universal property, the first part of Theorem\;\ref{1.1} is still true but the second part is not. In fact, there is no classic $2$-universal quaternary lattice over a dyadic local field (Proposition \ref{4.6}).

As a complement to Theorem\;\ref{1.1} and \cite[Remark 1.2]{xu_indefinite_2020}, we also obtain a necessary and sufficient condition for number fields which admit classic indefinite $1$-universal lattices (see Theorem\;\ref{6.1}).

\begin{thm}  \label{1.2}  A number field $F$ admits a classic integral lattice which is locally classic $1$-universal but not classic indefinite 1-universal if and only if $F$ has a quadratic unramified extension where all dyadic primes of $F$ split completely. In this case, there are infinitely many classes of such integral quadratic lattices over $F$.
\end{thm}

We also classify all quadratic fields with this property (see Theorem\;\ref{6.4}).

\

Unexplained notation and terminology are adopted from \cite{omeara_quadratic_1963}. Let $F$ be a number field with ring of integers $\mathcal{O}_{F} $, $\Omega_{F} $  the set of all primes of $ F $ and $ \infty_{F} $ the subset of archimedean primes. For any $ \mathfrak{p}\in \Omega_{F} $, we denote by $ F_{\mathfrak{p}} $ the completion of $ F$ at $\frak p$ and by $F_\frak p^\times$ the set of non-zero elements of $F_\frak p$.
If $\frak p\in \infty_F$, we put $\mathcal O_{\frak p}=F_{\frak p}$.  If $ \mathfrak{p}\in \Omega_{F} \setminus \infty_F$, we denote by $ \mathcal{O}_{\mathfrak{p}} $ the valuation ring of
$F_{\mathfrak{p}} $ and $ \mathcal{O}_{\mathfrak{p}}^{\times} $ the group of units. By abuse of notation, sometimes we also write $\frak p$ for the maximal ideal of $\cO_{\frak p}$. The normalized discrete valuation of $F_{\frak p}$ is denoted by $\ord_{\frak p}$. We write $\pi_{\frak p}$ for a uniformizer of $F_{\frak p}$ and  fix a non-square unit $\Delta_\frak p\in \mathcal{O}_{\frak p}^\times$ such that the quadratic extension $F_\frak p(\sqrt{\Delta_\frak p})/F_\frak p$ is unramified. For $\alpha, \beta\in F_\frak p^\times$, we write $(\alpha, \beta)_\frak p$ for the Hilbert symbol.

Let  $V$ be a quadratic space over $F$ (resp. $F_\frak p$), by which we mean a finite dimensional vector space endowed with a non-degenerate symmetric bilinear form $$B: V\times V\to F \ (\text{resp.} \ F_\frak p)  \ \ \text{with the associated quadratic form} \ \  Q(x)=B(x,x),\, x\in V . $$ We denote by $\det(V)$ the \emph{determinant} (or \emph{discriminant} in the terminology of \cite[\S\;41.B, p.87]{omeara_quadratic_1963}) of $V$.  A two dimensional quadratic space is called a hyperbolic plane and  denoted by $\Bbb H$ if it contains two vectors $x, y$ satisfying $Q(x)=Q(y)=0$ and $B(x,y)=1$.  For a quadratic space $V$ over $F_{\frak p}$, let $S_\frak p(V)$ be its Hasse symbol.

  An $\cO_F$-lattice (resp. $\cO_{\frak p}$-lattice) $L$ in $ V $ is a finitely generated $\mathcal{O}_{F}$-module (resp. $ \mathcal{O}_{\mathfrak{p}} $-module for $ \mathfrak{p}\in \Omega_{F} \setminus \infty_F$) inside $V$.   If $FL=V$ (resp. $F_\frak p L=V$), we say that $L$ is a lattice on $V$. If the quadratic space $V$ has  an orthogonal basis  $ x_{1},\ldots, x_{n} $ with $ Q(x_{i})=a_{i}$ for $1\leq i\leq n$, we simply write $[a_{1},\ldots,a_{n}] $ for $V$.  A lattice with such an orthogonal basis will be denoted by $  \langle a_{1},\ldots,a_{n}\rangle $.
 The \emph{scale} $\mathfrak{s}(L)$ and the \emph{norm} $\mathfrak{n}(L)$ of a lattice $L$ are defined as the fractional ideals
 $$ \mathfrak{s}(L) =B(L,L) \mathcal{O}_{F} \ (\text{resp.} \ B(L,L) \mathcal{O}_{\mathfrak{p}})  \ \ \ \text{and}  \ \  \  \mathfrak{n}(L) =Q(L)\mathcal{O}_{F} \ (\text{resp.} \ Q(L) \mathcal{O}_{\mathfrak{p}}). $$

\begin{defn} \label{int}  Let $L$ be an $\cO_F$-lattice (resp. $\cO_{\frak p}$-lattice).

(1)  We say  $L$ is \emph{integral}
 if $\mathfrak{n}(L)\subseteq \mathcal{O}_{F}$ (resp. $\mathfrak{n}(L)\subseteq   \mathcal{O}_{\mathfrak{p}}$).

 (2) A lattice $L$ is called \emph{classic integral} if $\mathfrak{s}(L)\subseteq \mathcal{O}_{F}$ (resp. $\mathfrak{s}(L)\subseteq   \mathcal{O}_{\mathfrak{p}}$).
 \end{defn}

 Let $W$ be another quadratic space over $F$ with the associated quadratic form $Q'$. We say that $W$ is represented by $V$ if there is an $F$-linear map $\sigma: W\to V$ such that $Q\circ \sigma=Q'$. Such a map $\sigma$ is called a \emph{representation} of $W$ into $V$. For two lattices $L$ and $M$, we say that $L$ is represented by $M$, written as $L\to M$, if there is a representation $$\sigma: FL\to FM \ \ \ \text{ such that } \ \ \ \sigma L\subseteq M. $$
The \emph{orthogonal group} $O(V)$ and the \emph{special orthogonal group} $O^+(V)$ of $V$ are defined by $$ O(V)=\{ \sigma \in GL(V): \ Q\circ \sigma=Q\}  \ \ \ \text{and} \ \ \ O^+(V)= \{ \sigma \in O(V): \ \det(\sigma)=1 \}\,. $$ The symmetry $\tau_u \in O(V)$ with respect to a vector $u\in V$ such that $Q(u)\neq 0$ is defined by
$$ \tau_u(x) = x-\frac{2B(x, u)}{Q(u)} u; \ \ \ \forall x\in V . $$
Write
$$O(L)= \{ \sigma\in O(V): \ \sigma L= L \} \ \ \ \text{and} \ \ \ O^+(L)=\{ \sigma \in O^+(V): \ \sigma L= L \} $$ for a given lattice $L$ on $V$.
Denote by $$\theta: \ O(V)\rightarrow F^\times / (F^\times)^2  $$  the spinor norm map of $V$ (cf. \cite[p.137, \S\;55]{omeara_quadratic_1963}).
For each prime $\frak p\in\Omega_F$, let $$V_{\frak p}=F_{\frak p}\otimes_F V \ \ \ \text{  and  } \ \ \  L_{\frak p}=\cO_{\frak p}\otimes_{\cO_F}L .$$ Then $O(V_{\frak p}),\,O^+(V_{\frak p})$, $O(L_{\frak p}),\,O^+(L_{\frak p})$,  $\theta_\frak p$ and $\tau_{u_\frak p}$ for $u_\frak p\in V_\frak p$ with $Q(u_\frak p)\neq 0$ are defined similarly.
The adelic groups of $O(V)$ and $O^+(V)$ are denoted by $O_\Bbb A(V)$ and $O^+_\Bbb A(V)$ respectively. They act naturally on a given lattice $L$ on $V$.
The orbit of $L$ under the action of $O_\Bbb A(V)$ (or $O^+_\Bbb A(V)$) is called the \emph{genus} of $L$ and denoted by $\gen(L)$ (see \cite[\S 102 A]{omeara_quadratic_1963}).

\medskip

 \begin{defn}\label{1.3}  Let $k\ge 1$ be an integer.

 (1) A quadratic space over $F$ (resp. $F_\frak p$) is called \emph{$k$-universal} if it represents all $k$-dimensional quadratic spaces over $F$ (resp. $F_\frak p$). When $k=1$, we simply say universal instead of $1$-universal.

(2)  An integral (resp. classic integral) $\mathcal{O}_{\mathfrak{p}}$-lattice is called \emph{$k$-universal (resp. classic $k$-universal)} if it represents all integral (resp. classic integral) $\mathcal{O}_{\mathfrak{p}}$-lattices of rank $k$ for $ \mathfrak{p}\in \Omega_{F} \setminus \infty_F$.

(3)  An integral (resp. classic integral) $\mathcal{O}_{F}$-lattice is called \emph{indefinite (resp. classic  indefinite) $k$-universal} if it represents all integral (resp. classic integral) $\mathcal{O}_{F}$-lattices of rank $k$.

(4) An integral (resp. classic integral)  $\mathcal{O}_{F}$-lattice $L$ is called \emph{$k$-LNG (resp. classic  $k$-LNG)} if $L_{\frak p}=\mathcal{O}_{\frak p}\otimes_{\mathcal{O}_F} L$ is $k$-universal (resp. classic $k$-universal) over $\mathcal{O}_{\frak p}$ for all $\frak p\in \Omega_F$ but $L$ is not indefinite $k$-universal (resp. classic indefinite $k$-universal).
\end{defn}

The paper is organized as follows. We first study $k$-universal spaces over a local field in Section\;\ref{sec2}. A consequence of the main result of this section is the first part of Theorem \ref{1.1}. In Section\;\ref{sec3}, we establish a local analogue of the Conway-Schneeberger theorem and determine all $k$-universal lattices over non-dyadic local fields in terms of Jordan decompositions. Over  dyadic local fields, 2-universal quaternary lattices are classified in Section\;\ref{sec4}. We show in Section\;\ref{sec5} that these results lead to a proof of the second part of  Theorem \ref{1.1}. In Section\;\ref{sec6}, we give a proof of Theorem \ref{1.2} and determine all quadratic fields which admit classic $1$-LNG lattices.

%

\section{On $k$-universal spaces}\label{sec2}

Before  studying the universal theory for integral quadratic forms, one needs to understand the corresponding theory for quadratic spaces. By the Hasse-Minkowski Theorem (see \cite[66:3.Theorem]{omeara_quadratic_1963}),  the $k$-universal property of a quadratic space over $F$ can be detected at each local completion $F_{\frak p}$ of $F$.

For convenience, we recall some known results about quadratic spaces.  It is well-known that the isometry class of any quadratic space over a local field is determined completely by its dimension,  determinant and Hasse symbol (see \cite[63:20 Theorem]{omeara_quadratic_1963}). These three invariants are independent of each other except in the one dimensional or the hyperbolic plane case.

\begin{thm} \label{ind} (\cite [63:22 Theorem]{omeara_quadratic_1963}) Let $V$ be a quadratic space over $F_\frak p$ with $\frak p\in \Omega_F\setminus \infty_F$. If $\dim(V)>1$ and  $V$ is not a hyperbolic plane, then there is a quadratic space $V'$ over $F_\frak p$ such that
$$ \dim(V')=\dim(V), \ \ \  \det(V')=\det(V) \ \ \ \text{and} \ \ \ S_\frak p(V')= - S_\frak p(V) . $$
\end{thm}

For representation of quadratic spaces, one has the following result.

\begin{thm}\label{rep} (\cite [63:21 Theorem]{omeara_quadratic_1963}) Let $U$ and $V$ be quadratic spaces over $F_\frak p$ with $\frak p\in \Omega_F\setminus \infty_F$. Write $\nu=\dim(V)-\dim(U) \geq 0$. Then $U \to V$ if and only if $\nu\geq 3$ or
$$ \begin{cases} U\cong V \ \ \ & \text{when $\nu=0$}\,, \\
U\perp [ \det(U) \cdot \det(V)] \cong V \ \ \ & \text{when $\nu=1$}\,, \\
U \perp \Bbb H \cong V  \ \ \ & \text{when $\nu=2$ and $\det(U)=-\det(V)$}.
\end{cases} $$
\end{thm}

The main result of this section is the following theorem.

\begin{thm}\label{2.1} Fix  $\frak p\in \Omega_F$ and let $V$ be a quadratic space over $F_\frak p$ .

\medskip

When $\frak p\in \Omega_F\setminus \infty_F$, $V$ is $k$-universal if and only if one of the following conditions holds:

(1) $\dim(V)\geq k+3$;

(2) $k=1$ and $V$ is isotropic of dimension $2$ or $3$;

(3) $k=2$ and $V\cong \Bbb H\perp \Bbb H$.

\medskip

When $\frak p$ is a real prime,  $V$ is $k$-universal  if and only if both the positive index and the negative index of $V$ are $\ge k$.

\medskip

When $\frak p$ is a complex prime,  $V$ is $k$-universal  if and only if $\dim(V)\geq k$.

\end{thm}

\begin{proof} First consider the case $\frak p\in\Omega_F\setminus\infty_F$. The sufficiency follows from Theorem \ref{rep} and the universal property of isotropic spaces.

\underline{Necessity.} For $k=1$, one only needs to prove that $V$ is not universal if $\dim (V) =3$ and $V$ is anisotropic by \cite[63:16.Example]{omeara_quadratic_1963}. Indeed, we claim that $V$ does not represent $-\det(V)$ in this case. Suppose not. Then the quadratic space $V\perp [ \det(V) ]$ is isotropic with discriminant 1. By \cite[63:18.Remark]{omeara_quadratic_1963}, one concludes that
$$V\perp [ \det(V) ] \cong \Bbb H \perp \Bbb H\,. $$ By computing the Hasse symbols of both sides with \cite[58:3.Remark]{omeara_quadratic_1963},  one obtains that $S_\frak p (V) =(-1, -1)_\frak p$. Therefore $V$ is isotropic by \cite[58:6]{omeara_quadratic_1963} and a contradiction is derived.

\

For $k=2$, one has $\dim(V)\geq 2$. Since there are more than two isometric classes of quadratic spaces over $F_\frak p$, one obtains that $\dim(V)\geq 3$ by Theorem \ref{rep}. We further claim that $\dim(V)\geq 4$. Suppose $\dim(V)=3$. Let $W$ be a binary quadratic space over $F_\frak p$ satisfying $W\not \cong \Bbb H$. Since $W\to V$, one has
$$V\cong W\perp [\det(V)\cdot \det(W)]$$ by Theorem \ref{rep}. By Theorem \ref{ind}, there is a binary quadratic space $W'$ over $F_\frak p$ such that $\det(W')=\det(W)$ and $S_\frak p(W')=-S_\frak p(W)$. Since $W' \to V$ as well, one has
$$V\cong W'\perp [\det(V)\cdot \det(W')]$$ by Theorem \ref{rep}.
By Witt cancellation (\cite[42:16.Theorem] {omeara_quadratic_1963}), one obtains that $W\cong W'$. This is a contradiction and the claim follows. Now one only needs to consider $\dim(V)=4$. Since $\Bbb H\to V$, one can further assume that $V=\Bbb H \perp U$ where $U$ is a binary quadratic space. Suppose that $U$ is not a hyperbolic plane. Then there is a binary quadratic space $U'$ over $F_\frak p$ such that $\det(U')=\det(U)$ and $S_\frak p(U')=-S_\frak p(U)$ by Theorem \ref{ind}. Therefore
$$\det(U')\cdot \det(V)=\det(U)\cdot \det(V)=-1.$$ Since $U'\to V$, one obtains that $V\cong U'\perp \Bbb H$ by Theorem \ref{rep}. By Witt cancellation, one concludes that $U\cong U'$ and a contradiction is derived. This implies that $U$ is a hyperbolic plane as desired.

\

For $k\geq 3$, one has $\dim(V)\geq k$. Since there are more than two isometric classes of quadratic spaces over $F_\frak p$, one obtains that $\dim(V)\geq k+1$ by Theorem \ref{rep}.

Suppose $\dim(V)=k+1$. By Theorem \ref{ind}, there are two quadratic spaces $U$ and $U'$ over $F_\frak p$ such that $$ \dim(U')=\dim(U)=k, \ \ \ \det(U')=\det(U) \ \ \ \text{ and } \ \ \ S_\frak p(U')=-S_\frak p(U) . $$
Since both $U$ and $U'$ are represented by $V$, one has
$$ V \cong U\perp [\det(V)\cdot \det(U)] \cong U'\perp [\det(V) \cdot \det(U')] $$ by Theorem \ref{rep}.
By Witt cancellation, one obtains $U\cong U'$. A contradiction is derived.

Suppose $\dim(V)=k+2$.  Since $\Bbb H \to V$, there is a quadratic space $U$ of dimension $k$ such that $V\cong \Bbb H\perp U$.
 By Theorem \ref{ind}, there is a quadratic space $U'$ over $F_\frak p$ such that $$\dim(U')=\dim(U), \ \ \ \det(U')=\det(U) \ \ \ \text{ and } \ \ \ S_\frak p(U')=-S_\frak p(U) . $$
 This implies that
$$\det(U')\cdot \det(V)=\det(U)\cdot \det(V)=-1.$$ Since $U' \to V$, one obtains that $V\cong U'\perp \Bbb H$ by Theorem \ref{rep}. By Witt cancellation, one concludes that $U\cong U'$ and a contradiction is derived.

Therefore one concludes that $\dim(V)\geq k+3$ as desired.

 \

When $\frak p$ is a real prime, the result follows from \cite[61:1.Theorem]{omeara_quadratic_1963} (for the necessity, consider the representation of the quadratic spaces $I_k$ and $-I_k$).

\

When $\frak p$ is a complex prime, the result follows from the fact that (non-degenerate) quadratic spaces of the same dimension over the complex numbers are all isomorphic.
\end{proof}

 An immediate consequence of Theorem \ref{2.1} is the following result.

\begin{cor} \label{2.3}  If $k\geq 3$, then there are no (classic) $k$-LNG lattices over any number field $F$. In other words, every (classic) integral $\mathcal O_F$-lattice $L$ such that $L_\frak p=L\otimes_{\cO_F} \cO_\frak p$ is (classic) $k$-universal over $\cO_\frak p$ for all $\frak p\in \Omega_F$ is (classic) indefinite $k$-universal over $\cO_F$.
\end{cor}

\begin{proof}  Since $L_\frak p$ is $k$-universal over $\cO_\frak p$ for all $\frak p\in \Omega_F$, one obtains that $F_\frak pL_\frak p$ is a $k$-universal quadratic space over $F_\frak p$. In particular, the quadratic space $FL$ is indefinite. Since $k\geq 3$, one concludes that $$\mathrm{rank}(L)=\dim(FL)=\dim(F_\frak pL_\frak p)\geq k+3$$ by Theorem \ref{2.1}. By \cite[p.135, line 2-4]{hsia_indefinite_1998}, $L$ represents all $\cO_F$-lattices of rank $k$ whose localization are represented by $L_\frak p$ over $\cO_{\frak p}$ for all $\frak p\in \Omega_F$. Since $L_\frak p$ is $k$-universal over $\cO_\frak p$ for all $\frak p\in \Omega_F$, one concludes that $L$ is indefinite $k$-universal over $\cO_F$. The argument is also valid for classic $\mathcal O_F$-lattices.
\end{proof}

\section{On $k$-universal lattices over non-dyadic local fields}\label{sec3}

In this section, we fix a non-dyadic prime $\frak p\in\Omega_F$, so that $F_\frak p$ is a non-dyadic local field.
Our goal is to  determine all $k$-universal lattices over  $\mathcal{O}_\frak p$. (Notice that the notions of integral lattices and classic lattices coincide over $\cO_{\frak p}$, since $2\in\cO_{\frak p}^\times$.)

\medskip

Recall that $\pi_{\frak p}$ denotes a uniformizer of $F_{\frak p}$ and  $\Delta_\frak p\in \mathcal{O}_{\frak p}^\times$ is chosen such that $F_\frak p(\sqrt{\Delta_\frak p})/F_\frak p$ is a quadratic unramified extension.

\begin{lem}\label{3.1}  The lattice $\langle \pi_\frak p, -\Delta_\frak p \pi_\frak p \rangle$ is the $ \mathcal{O}_{\frak p} $-maximal lattice on the quadratic space $[\pi_\frak p, -\Delta_\frak p \pi_\frak p]$.
\end{lem}
\begin{proof} Let $\{x, y\}$ be an orthogonal basis of  $[\pi_\frak p, -\Delta_\frak p \pi_\frak p]$ with $Q(x)=\pi_\frak p$ and $Q(y)=-\Delta_\frak p \pi_\frak p$. Since the quadratic space $[\pi_\frak p, -\Delta_\frak p \pi_\frak p]$ is anisotropic, an element $\alpha x+ \beta y$ with $\alpha, \beta \in F_\frak p$ belongs to the unique $\mathcal{O}_{\frak p} $-maximal lattice if and only if
$$Q(\alpha x+ \beta y)= \alpha^2 \pi_\frak p - \beta^2 \Delta_\frak p \pi_\frak p \in  \mathcal{O}_{\frak p} $$
by \cite[91:1. Theorem]{omeara_quadratic_1963}. If $\ord_\frak p(\alpha)\neq \ord_\frak p(\beta)$, then
$$ \ord_\frak p(Q(\alpha x+ \beta y))= \min \{ \ord_\frak p(\alpha^2 \pi_\frak p), \ord_\frak p(\beta^2 \Delta_\frak p \pi_\frak p) \} \geq 0 .$$ This implies that $\alpha, \beta \in \mathcal{O}_{\frak p} $. Suppose  $\ord_\frak p(\alpha)= \ord_\frak p(\beta) <0$. There are $\xi, \eta\in  \mathcal{O}_{\frak p}^\times $ such that $\xi^2-\Delta_\frak p \eta^2 \in \pi_{\frak p}\mathcal{O}_{\frak p}$ by removing the denominators. Then $\Delta_\frak p$ is a square by \cite[63:1.Local Square Theorem]{omeara_quadratic_1963}. A contradiction is derived.
\end{proof}

Let us call a set $S$ of integral lattices of rank $k$ a \emph{testing set} for $k$-universality if for every integral lattice representing all members in $S$ is $k$-universal. It is not difficult to provide a finite testing set for $k$-universality over  $\mathcal{O}_{\frak p}$.  The crucial part of a local analogue of the Conway-Schneeberger theorem is to find a \emph{minimal} testing set, i.e., a testing set such that none of its proper subsets is enough for testing the $k$-universality.

\begin{prop}\label{3.2}
 A minimal testing set for $k$-universality over $\mathcal{O}_{\frak p}$ is given as follows:

 When $k=1$, the set consists of the following $4$ lattices
 $$\langle 1 \rangle;  \  \langle \Delta_\frak p \rangle;  \ \langle \pi_\frak p \rangle;  \  \langle \Delta_\frak p \pi_\frak p \rangle . $$

 When $k=2$, the set consists of the following $7$ lattices
$$  \langle 1,-1\rangle ; \  \langle 1,-\Delta_\frak p \rangle ;  \  \langle \pi_\frak p, -\Delta_\frak p \pi_\frak p \rangle ;  \  \langle 1,-\pi_\frak p  \rangle ;  \  \langle \Delta_\frak p,-\Delta_\frak p\pi_\frak p \rangle ; \ \langle 1,-\Delta_\frak p\pi_\frak p \rangle ; \  \langle\Delta_\frak p,-\pi_\frak p\rangle . $$

When $k\geq 3$, the set consists of the following $8$ lattices
$$ \langle 1, \cdots, 1 \rangle;  \  \langle \cdots,  -\Delta_\frak p, \pi_\frak p , -\Delta_\frak p \pi_\frak p \rangle;  \  \langle 1, \cdots, 1, \Delta_\frak p \rangle;  \ \langle \cdots, -1, \pi_\frak p, -\Delta_\frak p \pi_\frak p \rangle ; $$
$$ \langle 1, \cdots, -1, -\pi_\frak p \rangle;  \ \langle 1, \cdots, -\Delta_\frak p, -\Delta_\frak p \pi_\frak p  \rangle ;   \ \langle 1, \cdots, -1, -\Delta_\frak p \pi_\frak p \rangle ;  \ \langle 1,  \cdots, -\Delta_\frak p, -\pi_\frak p \rangle $$
where the elements not showing up in the notations are $1$.
\end{prop}
\begin{proof}
Let us first show that the set of lattices given in the proposition are indeed testing sets for $k$-universality in the respective cases.

Every integral $\mathcal{O}_{\frak p}$-lattice of rank $k$ is contained in an $\mathcal{O}_{\frak p}$-maximal lattice of rank $k$ by \cite[82:18]{omeara_quadratic_1963}, and all $\mathcal{O}_{\frak p}$-maximal lattices in a fixed quadratic space are isometric by \cite[91:2. Theorem]{omeara_quadratic_1963}. So it is sufficient to prove that the lattices listed in the proposition  are the $\mathcal{O}_{\frak p}$-maximal lattices from all possible $k$-dimensional quadratic spaces.

There are only $4$ one-dimensional quadratic spaces over $F_\frak p$ up to isometry. The lattices listed for $k=1$ are the corresponding $\mathcal{O}_{\frak p}$-maximal lattices. The result can also follow from \cite[Lemma 2.2]{xu_indefinite_2020}.

There are precisely  $7$ binary quadratic spaces over $F_\frak p$ up to isometry by \cite[63:9]{omeara_quadratic_1963} and Theorem \ref{ind}. One concludes that the lattices listed for $k=2$ come from $7$ different quadratic spaces by computing their discriminants and Hasse symbols. The listed lattices are also $\mathcal{O}_{\frak p}$-maximal by \cite[82:19]{omeara_quadratic_1963} and Lemma \ref{3.1}.

For $k\geq 3$, there are exactly $8$ quadratic spaces of dimension $k$ over $F_\frak p$ up to isometry by \cite[63:9]{omeara_quadratic_1963} and Theorem \ref{ind}. Since the listed lattices of rank $k$ are from different quadratic spaces by computing discriminants and Hasse symbols, one only needs to show that  $\langle \cdots,  -\Delta_\frak p, \pi_\frak p , -\Delta_\frak p \pi_\frak p \rangle$ and  $\langle \cdots, -1, \pi_\frak p, -\Delta_\frak p \pi_\frak p \rangle $ are $\mathcal{O}_{\frak p}$-maximal lattices by \cite[82:19]{omeara_quadratic_1963}.
Suppose that $L$ is an integral $\mathcal{O}_{\frak p}$-lattice such that
$$ L \supseteq  \langle \cdots,  -\Delta_\frak p, \pi_\frak p , -\Delta_\frak p \pi_\frak p \rangle \ \ \  \ \text{or} \ \ \  \ L \supseteq \langle \cdots, -1, \pi_\frak p, -\Delta_\frak p \pi_\frak p \rangle . $$ Since both $\langle \cdots,  -\Delta_\frak p \rangle$ and $\langle \cdots, -1\rangle$ are unimodular, there is a binary lattice $L_1$ on $[\pi_\frak p, -\Delta_\frak p \pi_\frak p]$ with $\frak n(L_1) \subseteq \mathcal{O}_{\frak p}$ such that
$$ L= \langle \cdots,  -\Delta_\frak p \rangle \perp L_1 \ \ \ \ \text{or} \ \ \ \ L=\langle \cdots, -1\rangle \perp L_1 $$ by \cite[82:15a]{omeara_quadratic_1963}. Then $L_1\supseteq \langle  \pi_\frak p , -\Delta_\frak p \pi_\frak p \rangle$. By Lemma \ref{3.1}, one concludes that $L_1= \langle  \pi_\frak p , -\Delta_\frak p \pi_\frak p \rangle$ and
$$L = \langle \cdots,  -\Delta_\frak p, \pi_\frak p , -\Delta_\frak p \pi_\frak p \rangle \ \ \  \ \text{or} \ \ \  \ L = \langle \cdots, -1, \pi_\frak p, -\Delta_\frak p \pi_\frak p \rangle  $$ as desired.

Finally, we show that the sets of  integral $\mathcal{O}_{\frak p} $-lattices listed above are minimal.

Indeed,
$$ \begin{cases}
 \langle \Delta_\frak p,   -\pi_\frak p, \Delta_\frak p \pi_\frak p \rangle \ \text{represents} \   \langle \Delta_\frak p \rangle,  \ \langle \pi_\frak p \rangle,  \  \langle \Delta_\frak p \pi_\frak p \rangle \  \text{but does not represent}  \ \langle 1 \rangle \\
 \langle 1,   \pi_\frak p, -\Delta_\frak p \pi_\frak p \rangle \ \text{represents} \   \langle 1 \rangle,  \ \langle \pi_\frak p \rangle,  \  \langle \Delta_\frak p \pi_\frak p \rangle \  \text{but does not represent}  \ \langle \Delta_\frak p \rangle \\
  \langle -1,   \Delta_\frak p, \Delta_\frak p \pi_\frak p \rangle \ \text{represents} \   \langle 1 \rangle,  \ \langle \Delta_\frak p \rangle,  \  \langle \Delta_\frak p \pi_\frak p \rangle \  \text{but does not represent}  \ \langle \pi_\frak p \rangle \\
 \langle 1, -\Delta_\frak p, \pi_\frak p \rangle \ \text{represents} \ \langle 1 \rangle,  \  \langle \Delta_\frak p \rangle,  \ \langle \pi_\frak p \rangle    \ \text{but does not represent}  \  \langle \Delta_\frak p \pi_\frak p \rangle  \\
 \end{cases} $$
by \cite[63:15.Example and 63:17]{omeara_quadratic_1963}. It follows that the set of integral $\mathcal{O}_{\frak p}$-lattices listed above for $k=1$ is minimal.

For $k=2$, we take a binary integral $\mathcal{O}_{\frak p}$-lattice $L$ in the above list.

If $L=\langle 1,-1\rangle$, then $ \langle 1,-\Delta_\frak p \rangle \perp  \langle \pi_\frak p, -\Delta_\frak p \pi_\frak p \rangle$ does not represent $L$ by \cite[63:17]{omeara_quadratic_1963}  but represents the remaining six lattices in the list by \cite[Theorem 1]{omeara_integral_1958} and \cite[63:15.Example]{omeara_quadratic_1963}.

Otherwise, there is another binary integral $\mathcal{O}_{\frak p}$-lattice $L'$ in the above list satisfying $$F_\frak pL'\not \cong F_\frak pL  \ \ \  \text{and}   \ \ \ \det(F_\frak pL)=\det(F_\frak p L') . $$
We claim that $L'\perp \langle 1,-1\rangle$ does not represent $L$. Suppose not. Then $F_\frak p L'\perp \Bbb H$ represents $F_\frak p L$ as quadratic spaces.  This implies that $$F_\frak p L\perp \Bbb H \cong F_\frak p L'\perp \Bbb H$$ by Theorem \ref{rep}. By Witt cancellation (\cite[42:16.Theorem] {omeara_quadratic_1963}), one obtains $F_\frak pL' \cong F_\frak pL $. A contradiction is derived and the claim follows. By \cite[Theorem 1]{omeara_integral_1958} and Theorem \ref{rep}, one can verify that $L'\perp \langle 1,-1\rangle$
represents the other six lattices. Therefore, the set of binary lattices in the given list for $k=2$ is minimal.

For any integral $\mathcal{O}_{\frak p}$-lattice $L$ of  rank $k\geq 3$ in our list, there is another integral $\mathcal{O}_{\frak p}$-lattice $L'$ of rank $k$ in the  list satisfying $$F_\frak pL'\not \cong F_\frak pL  \ \ \  \text{and}   \ \ \ \det(F_\frak pL)=\det(F_\frak p L') . $$
Then $L'\perp \langle 1,-1\rangle$ does not represent $L$ but represents the rest of seven lattices in the given list by the same arguments as above. Therefore the testing set given for $k\ge 3$ is minimal.
\end{proof}

Since the lattices listed in Proposition \ref{3.2} are all $\mathcal{O}_{\frak p}$-maximal, one concludes that the minimal sets are unique up to isometry by \cite[91:2.Theorem]{omeara_quadratic_1963}. By applying Proposition \ref{3.2}, we can determine all $k$-universal lattices over the non-dyadic field $F_{\frak p}$ in terms of Jordan splittings. The  $k=1$ case has been done in \cite[Prop.\;2.3]{xu_indefinite_2020}. For convenience, we treat the cases $k=2$ and $k\geq 3$ separately.

\begin{prop}\label{3.3}
	Let $ M $ be an integral $ \mathcal{O}_{\frak p} $-lattice with the following Jordan splitting $$ M=M_{1}\perp M_{2}\perp \ldots \perp M_{t} $$ where $M_1$ is unimodular. Then
$ M $ is $2$-universal if and only if one of the following conditions holds:

	(a)  $ \rank (M_{1})=3 $ and $M_{2} $ is $\frak p$-modular with $ \rank (M_{2})\ge 2 $.
	
	(b) $M_{1}\cong \langle 1,1,1,1 \rangle$.
	
	(c) $ M_{1}\cong \langle 1,1,1,\Delta_\frak p \rangle$ and $ M_{2} $ is $ \frak p$-modular with $\rank (M_2)\geq 1$.
	
	(d) $\rank (M_1) \geq 5$.
\end{prop}
\begin{proof} \underline{Necessity}.  Since $M$ represents both $ \langle 1,-1\rangle$ and $ \langle 1,-\Delta_\frak p \rangle$, we have $\rank (M_1)\geq 3$ by \cite[Theorem 1]{omeara_integral_1958} and  Theorem \ref{rep}.
	
	When $ \rank (M_{1})=3 $, the space $F_\frak p M_1$ is not  $2$-universal by Theorem\;\ref{2.1}. Since $M$ represents all the  7 binary lattices listed in Proposition\;\ref{3.2},  $M_2$ must be $\frak p$-modular and $F_\frak pM_1\perp F_\frak p M_2$ is a $2$-universal quadratic space by \cite[Theorem 1]{omeara_integral_1958}. Applying Theorem \ref{2.1}, one concludes that $\rank(M_2)\geq 2$.
	
	When $\rank(M_1)=4$, one obtains $M_1\cong \langle 1,1,1,1 \rangle$ or $\langle 1,1,1,\Delta_\frak p \rangle$ by \cite[92:1]{omeara_quadratic_1963}. Suppose $M_1\cong \langle 1,1,1,\Delta_\frak p \rangle$. Then $F_\frak pM_1$ is not a $2$-universal quadratic space by Theorem \ref{2.1}. This implies that $M_1$ cannot represent all the binary lattices listed in Proposition\;\ref{3.2} by \cite[Theorem 1]{omeara_integral_1958}. Therefore, $M_2$ is $\frak p$-modular with $\rank(M_2)\geq 1$ by \cite[Theorem 1]{omeara_integral_1958}.

\medskip	
	
\underline{Sufficiency}. It suffices to verify that $M$ represents all the binary lattices listed in Proposition\;\ref{3.2} if one of the conditions $(a)$, $(b)$, $(c)$ or $(d)$ holds.  	
	
If condition $(a)$ holds, then $M_1$ represents both $  \langle 1,-1\rangle $ and  $\langle 1,-\Delta_\frak p \rangle$ by \cite[92:1]{omeara_quadratic_1963}. Since $M_1$ represents all elements in $ \mathcal{O}_{\mathfrak{p}}^{\times} $ by \cite[92:1b]{omeara_quadratic_1963}, one concludes that $M_1\perp M_2$ represents the other binary lattices listed in Proposition\;\ref{3.2} by \cite[Theorem 1]{omeara_integral_1958} and Theorem \ref{rep}.

If condition $(b)$ holds, then $F_\frak p M_1\cong \Bbb H\perp \Bbb H$. In this case $M_1$ represents all the binary lattices listed in Proposition\;\ref{3.2} by  \cite[Theorem 1]{omeara_integral_1958} and Theorem\;\ref{2.1}.

If condition $(c)$ holds, then $M_1$ represents both $  \langle 1,-1\rangle $ and  $\langle 1,-\Delta_\frak p \rangle$ and $M_1\perp M_2$ represents the other binary lattices in the testing set by \cite[Theorem 1]{omeara_integral_1958} and Theorem \ref{rep}.

If condition $(d)$ holds, then $M_1$ represents all the binary lattices in the list by \cite[Theorem 1]{omeara_integral_1958} and Theorem \ref{rep}.
 \end{proof}

Finally, we determine all $k$-universal lattices for $k\geq 3$.

\begin{prop}\label{3.4}
Suppose $k\ge 3$.	Let $ M $ be an integral $ \mathcal{O}_{\frak p} $-lattice with the following Jordan splitting $$ M=M_{1}\perp M_{2}\perp \ldots \perp M_{t} $$ where $M_1$ is unimodular. Then $ M $ is $k$-universal if and only if one of the following conditions holds:

	(a)  $ \rank (M_{1})=k+1 $ and $M_{2} $ is $\frak p$-modular with $ \rank (M_{2})\ge 2 $.
	
	(b) $\rank (M_{1}) =k+2 $ and $ M_{2} $ is $ \frak p$-modular with $\rank (M_2)\geq 1$.
	
	(c) $\rank (M_1) \geq k+3 $.
\end{prop}
\begin{proof} \underline{Necessity}.  Since $M$ represents the rank $k$ lattices $ \langle 1, \cdots, 1 \rangle$ and $\langle 1, \cdots, 1, \Delta_\frak p \rangle$, we have $\rank (M_1)\geq k+1$ by \cite[Theorem 1]{omeara_integral_1958} and Theorem \ref{rep}.
	
	When $ \rank (M_{1})=k+1$, the quadratic space  $F_\frak pM_1$ is not  $k$-universal by Theorem\;\ref{2.1}. Since $M$ represents the 8 lattices of rank $k$ listed
in Proposition\;\ref{3.2}, it follows from \cite[Theorem 1]{omeara_integral_1958} that $M_2$ is $\frak p$-modular and $F_\frak pM_1\perp F_\frak p M_2$ is a $k$-universal quadratic space. Then Theorem\;\ref{2.1} implies $\rank (M_2)\geq 2$.
	
	When $\rank(M_1)=k+2$, the same argument as above shows that $M_2$ is $\frak p$-modular with $\rank (M_2)\geq 1$.
	
\medskip	
	
\underline{Sufficiency}. We need to verify that $M$ represents all the lattices of rank $k$ listed in Proposition\;\ref{3.2} under one of the conditions $(a)$, $(b)$ and $(c)$.  	
	
If condition $(a)$ or $(b)$ holds, then $M_1$ represents all unimodular lattices of rank less than $k+1$ by \cite[92:1 and 92:1b]{omeara_quadratic_1963}. Therefore
 $M_1\perp M_2$ represents all the lattices of rank $k$ listed in Proposition\;\ref{3.2} by \cite[Theorem 1]{omeara_integral_1958} and Theorem \ref{rep}.

If condition $(c)$ holds, the same method shows that  $M_1$ represents all the lattices of rank $k$ in the testing set given in Proposition\;\ref{3.2}.
 \end{proof}

\section{On 2-universal quaternary lattices over dyadic local fields}\label{sec4}

Determining $k$-universal lattices over dyadic local fields is much more difficult than over non-dyadic fields. Based on O'Meara's work \cite{omeara_integral_1958},
a classification of (classic) $k$-universal lattices using Jordan splittings has been obtained in \cite{HeHu1} for unramified dyadic local fields. By introducing the concept of
bases of norm generators (BONGs in short), Beli has recently developed an integral representation theory over general dyadic local fields (cf. \cite{beli_representations_2006} and \cite{beli_representations_2019}). He classified 1-universal lattices over general dyadic fields in \cite{beli_universal_2020} by his theory. In a forthcoming work \cite{HeHu2},  all (classic) $k$-universal lattices for $k\ge 2$ will be determined by using BONGs.

In this paper our main concern is about the local-global principle for $k$-universality. In view of Corollary\;\ref{2.3}, we are satisfied with the determination of  quaternary 2-universal lattices over a general dyadic local field.

\medskip

Throughout this section, we fix a dyadic prime $\mathfrak{p}$ of the number field $F$, and we put $e= \ord_\frak p (2) $. For $ a\in \mathcal{O}_{\frak p}^{\times}$, define its \textit{quadratic defect} by $$ \mathfrak{d}(a)=\bigcap_{x\in F_\frak p}(a-x^{2})\mathcal{O}_{\frak p}$$
   as in \cite[\S\,63A.]{omeara_quadratic_1963}.
   Recall that we use $\Delta_{\frak p}$ to denote an element in $\cO_{\frak p}^\times$ such that $F_\frak p(\sqrt{\Delta_\frak p})$ is the unique quadratic unramified extension of $F_\frak p$. This $\Delta_{\frak p}$ satisfies $ \mathfrak{d}(\Delta_{\frak p})=4\mathcal{O}_{\frak p} $. We may and we shall assume that $\Delta_{\frak p}=1-4\rho_\frak p $ with $\rho_{\frak p}\in \cO_{\frak p}^\times$ (cf. \cite[p.251, \S\;93]{omeara_quadratic_1963}).

\medskip

The Hilbert symbol computation in the following lemma is an explicit version of \cite[Lemma 3]{hsia_spinor_1975}.

\begin{lem} \label{4.1} Let  $t$ be an odd integer with $1\leq t\leq 2e-1$ and $\sigma \in \mathcal{O}_{\frak p}^\times$. Then
$$\big(1-\sigma \pi_\frak p^t\,,\; 1-4\rho_\frak p \sigma^{-1} \pi_\frak p^{-t}\big)_\frak p =-1.$$
\end{lem}
\begin{proof} First, we have $\big(\Delta_\frak p\,,\, \sigma \pi_\frak p^t-4\rho_\frak p\big)_\frak p=-1$ by  \cite[63:11a]{omeara_quadratic_1963}. Therefore
$$ \begin{aligned}  &
\big(1-\sigma \pi_\frak p^t\,,\, 1-4\rho_\frak p \sigma^{-1} \pi_\frak p^{-t} \big)_\frak p =\bigg(1-\sigma \pi_\frak p^t\,,\, \frac{\sigma \pi_\frak p^t-4\rho_\frak p}{ \sigma \pi_\frak p^{t}}\bigg)_\frak p=\big(1-\sigma \pi_\frak p^t\,,\, \sigma \pi_\frak p^t-4\rho_\frak p\big)_\frak p \\
 = & - \big((1-\sigma \pi_\frak p^t)\Delta_\frak p\,,\, \sigma \pi_\frak p^t-4\rho_\frak p\big)_\frak p = - \big((1-\sigma \pi_\frak p^t)\Delta_\frak p\,,\,
 \Delta_\frak p -(1-\sigma \pi_\frak p^t )\big)_\frak p \\
= & -\left((1-\sigma \pi_\frak p^t)\Delta_\frak p\,,\, \big(\Delta_\frak p -(1-\sigma \pi_\frak p^t )\big)\Delta_\frak p\right)_\frak p=
-\bigg(\frac{1-\sigma \pi_\frak p^t}{\Delta_\frak p}\,,\, 1 - \frac{1-\sigma \pi_\frak p^t}{\Delta_\frak p}\bigg)_\frak p=-1 \end{aligned} $$
by \cite[57:10 and 63:11a]{omeara_quadratic_1963}.
\end{proof}

Let $N\frak p$ be the number of elements in the residue field of $F$ at $\frak p$. By \cite[63:9, 63:20.Theorem and 63:22.Theorem]{omeara_quadratic_1963}, there are $8(N\frak p)^e-1$ binary quadratic spaces over $F_\frak p$ up to isometry. One can classify binary $\mathcal{O}_{\frak p}$-maximal lattices accordingly.

   For any elements $\xi,\,\eta\in \cO_{\frak p}$ and $\gamma\in F_{\frak p}^\times$, we  write $ \gamma A(\xi,\eta) $ for the binary lattice $$ \mathcal{O}_{\frak p}x+\mathcal{O}_{\frak p}y  \ \  \ \text{ with } \ \ \  Q(x)=\gamma\xi,  \ B(x,y)=\gamma   \ \text{ and } \  Q(y)= \gamma\eta . $$

\begin{prop}\label{4.2} Any binary $\mathcal{O}_{\frak p}$-maximal lattice is isometric to one of the following lattices:

 Type  I:  $$\pi_\frak p^{-i} A(\pi_{\frak p}^i, \sigma \pi_\frak p^{i+1})  \ \ \ \text{ and } \ \ \ (1-4\rho_\frak p \sigma^{-1} \pi_\frak p^{-2i-1})\pi_\frak p^{-i} A(\pi_{\frak p}^i, \sigma \pi_\frak p^{i+1}) $$
with $0\leq i\leq e-1$ and $\sigma \in \mathcal{O}_{\frak p}^{\times}$.

 Type II: \  $\langle 1, \sigma \pi_\frak p \rangle $  and   $\langle \Delta_\frak p, \sigma \pi_\frak p \rangle $
with $\sigma \in \mathcal{O}_{\frak p}^{\times}$.

Type III: \  $ 2^{-1}A(2, 2\rho_\frak p) $ and  $ 2^{-1}\pi_\frak p A(2, 2\rho_\frak p)$.

Type IV: \ $2^{-1}A(0, 0)$.
\end{prop}
\begin{proof}There is only one isometry class of $\mathcal{O}_{\frak p}$-maximal lattices in a given quadratic space (cf. \cite[91:2. Theorem]{omeara_quadratic_1963}). So we only need to prove that the lattices listed above exhaust all possible binary quadratic spaces over $F_{\frak p}$.

\medskip

Suppose that $V$ is a binary quadratic space whose discriminant is a unit $\epsilon$.

If $-\epsilon$ is a square, then $V\cong \Bbb H$ and the lattice of Type IV is an $\mathcal{O}_{\frak p}$-maximal lattice by \cite[82:21]{omeara_quadratic_1963}.

If $-\epsilon \in \Delta_\frak p (\mathcal{O}_{\frak p}^{\times})^2$, there are two quadratic spaces with discriminant $\epsilon$ up to isometry by \cite[63:15.Example]{omeara_quadratic_1963}. The $\mathcal{O}_{\frak p}$-maximal lattice on the quadratic space which represents $1$ is the lattice $2^{-1}A(2, 2\rho_\frak p)$ in Type III by \cite[93:11]{omeara_quadratic_1963}. The  $\mathcal{O}_{\frak p}$-maximal lattice on the other quadratic space, which represents $\pi_\frak p$, is the lattice $ 2^{-1}\pi_\frak p A(2, 2\rho_\frak p)$ in Type III, by \cite[91:1.Theorem]{omeara_quadratic_1963} and the principle of domination in \cite{riehm_integral_1964}.

Otherwise we can assume that $$-\epsilon=1-\sigma \pi_\frak p^{2i+1} \ \ \ \text{ with } \ \ \ 0\leq i\leq e-1 \ \ \ \text{and} \ \ \ \sigma\in \mathcal{O}_{\frak p}^\times $$ by \cite[63:2]{omeara_quadratic_1963}. Up to isometry there are only two quadratic spaces  with  discriminant $\epsilon$ by Theorem \ref{ind}. By scaling if necessary, one can assume that $V$ represents $1$. Then the lattice  $\pi_\frak p^{-i} A(\pi_{\frak p}^i, \sigma \pi_\frak p^{i+1})$ in Type I is the $\mathcal{O}_{\frak p}$-maximal lattice on $V$ by \cite[91:1.Theorem]{omeara_quadratic_1963} and the principle of domination in \cite{riehm_integral_1964}. Lemma \ref{4.1} implies that $1-4\rho_\frak p \sigma^{-1} \pi_\frak p^{-2i-1}$ is not represented by $V$. So one concludes that the other quadratic space with  discriminant $\epsilon$ can be obtained by scaling $V$ with $1-4\rho_\frak p \sigma^{-1} \pi_\frak p^{-2i-1}$ by \cite[63:15.Example (iii)]{omeara_quadratic_1963}. Thus, the lattice $$(1-4\rho_\frak p \sigma^{-1} \pi_\frak p^{-2i-1})\pi_\frak p^{-i} A(\pi_{\frak p}^i, \sigma \pi_\frak p^{i+1})  \ \ \ \text{in Type I} $$ is the $\mathcal{O}_{\frak p}$-maximal lattice on this space.

\medskip

Now suppose that $V$ is a binary quadratic space with  discriminant $\epsilon\pi_\frak p$ for some $\epsilon\in \mathcal{O}_{\frak p}^\times$. By scaling if necessary, one can assume that $V$ represents $1$. Then the lattice $\langle 1, \epsilon \pi_\frak p \rangle$ in Type II is the  $\mathcal{O}_{\frak p}$-maximal lattice on $V$ by \cite[91:1.Theorem]{omeara_quadratic_1963}. Since $\Delta_\frak p$ is not represented by $V$ by \cite[63:11a]{omeara_quadratic_1963}, the other quadratic space with discriminant $\epsilon\pi_\frak p$ is  $[\Delta_\frak p, \Delta_\frak p \epsilon \pi_\frak p]$, according to \cite[ 63:15.Example(iii); 63:20.Theorem and 63:22.Theorem]{omeara_quadratic_1963}. The lattice $\langle \Delta_\frak p, \Delta_\frak p\epsilon \pi_\frak p \rangle$ in Type II is the $\mathcal{O}_{\frak p}$-maximal lattice on this space.
\end{proof}

The following result can be regarded as  a local analogue of the Conway-Schneeberger theorem for $2$-universal integral lattices over dyadic local fields.

\begin{cor}\label{4.3}  An integral $\cO_{\frak p}$-lattice is $2$-universal if and only if it represents all the lattices of type I, II, III and IV in Proposition $\ref{4.2}$.
\end{cor}

\begin{proof} Similarly as in the proof of Proposition\;\ref{3.2}, to test 2-universality of an integral lattice it suffices to check that it  represents the $\mathcal{O}_{\frak p}$-maximal lattices in all binary quadratic spaces. So the result is immediate from  Proposition \ref{4.2}.
\end{proof}

\begin{lem}\label{4.4} Suppose $M = 2^{-1} A(0, 0) \perp N$ where $N$ is an $\mathcal{O}_{\frak p}$-lattice with $\frak s(N) \subseteq 2^{-1} \frak p$. If $\frak n(N) \subseteq \frak p$, then $M$ cannot represent $ 2^{-1} A(2, 2\rho_\frak p) $.
\end{lem}
\begin{proof} Let $\{x, y\}$ be a basis for $ 2^{-1} A(0, 0) $ and $\{u, v\}$ be a basis for $ 2^{-1} A(2, 2\rho_\frak p) $. Suppose $M$ represents $ 2^{-1} A(2, 2\rho_\frak p) $. Then there are elements $\alpha_1, \beta_1$  and $\alpha_2, \beta_2$ in  $\mathcal{O}_{\frak p}^\times$ such that
$$ \begin{cases} u= \alpha_1 x+\beta_1 y + z_1 \\
v=\alpha_2 x+\beta_2 y +z_2 \end{cases} $$ with $z_1, z_2\in N$. Then
\begin{align*}
& \Delta_\frak p=1-4\rho_\frak p=4(B(u, v)^2-Q(u)Q(v))  \\
= \ &   (2B(z_1, z_2)+\alpha_1\beta_2+\alpha_2\beta_1)^2- 4\alpha_1\alpha_2\beta_1\beta_2 - 4 (\alpha_1\beta_1 Q(z_2)+\alpha_2\beta_2 Q(z_1)+Q(z_1)Q(z_2)) \\
= \ & (2B(z_1, z_2)+\alpha_1\beta_2-\alpha_2\beta_1)^2 + 8\alpha_2 \beta_1 B(z_1, z_2)- 4 (\alpha_1\beta_1 Q(z_2)+\alpha_2\beta_2 Q(z_1)+Q(z_1)Q(z_2)).
  \end{align*}
Since $\frak s(N) \subseteq 2^{-1} \frak p$, it follows that $2\alpha_2 \beta_1B(z_1, z_2)\in \frak p$. If $\frak n(N) \subseteq \frak p$, one concludes that $$\alpha_1\beta_1 Q(z_2)+\alpha_2\beta_2 Q(z_1)+Q(z_1)Q(z_2) \in \frak p . $$  This implies that $\Delta_\frak p$ is a square by the local square theorem, whence a contradiction. \end{proof}

Recall that the \textit{norm group}  of an $\cO_{\frak p}$-lattice $L$ is defined by $ \mathfrak{g}(L)=Q(L)+2\mathfrak{s}(L) $. We can now prove our main results in this section.

\begin{prop} \label{4.5} An integral quaternary $\mathcal{O}_{\frak p}$-lattice $M$ is $2$-universal if and only if $$M\cong 2^{-1}A(0, 0) \perp 2^{-1}A(0, 0) . $$
\end{prop}
\begin{proof}  \underline{Sufficiency}. By Corollary \ref{4.3}, one only needs to verify that $2^{-1}A(0, 0) \perp 2^{-1}A(0, 0)$ represents all the lattices listed in Proposition \ref{4.2}.  Since $2^{-1}A(0, 0)$ is universal, the lattices in Type II or IV are represented by $M$.  Let $L$ be a lattice in Type I or III. In order to apply \cite[Theorem 7.2]{riehm_integral_1964}, we scale $M$ and $L$ by 2 and call the resulting lattices $M'$ and $L'$ respectively.

If $L$ is in Type I, then we have
$$ \frak g(L')=\frak n(L')= \frak g (M')= \frak n(M')= 2 \mathcal{O}_{\frak p}$$  by \cite[Lemma 7.1]{riehm_integral_1964} and Proposition \ref{4.2}. Since ${L'}^{\sharp}$ is isomorphic to $$ 2^{-1}\pi_\frak p^{i} A(\pi_{\frak p}^i, \sigma \pi_\frak p^{i+1})  \ \ \ \text{or} \ \ \ 2^{-1}(1-4\rho_\frak p \sigma^{-1} \pi_\frak p^{-2i-1})\pi_\frak p^{i} A(\pi_{\frak p}^i, \sigma \pi_\frak p^{i+1}) $$ with $0\leq i\leq e-1$ and $$\frak n({L'}^{\sharp})= \frak p^{2i-e} \supset 2 \mathcal{O}_{\frak p}=\frak n(M') , $$ one concludes that $M'$ represents $L'$ by \cite[Theorem 7.2]{riehm_integral_1964} and Theorem \ref{2.1}.

If $L$ is in Type III, one only needs to consider the case $L\cong 2^{-1}\pi_\frak p A(2, 2\rho_\frak p)$ since
$$ 2^{-1} A(2, 2\rho_\frak p) \perp 2^{-1} A(2, 2\rho_\frak p) \cong 2^{-1} A(0, 0) \perp 2^{-1} A(0, 0) . $$
Since the norm group of $L'$ satisfies
$$ \frak g(L')= \frak n(L')= \frak p^{e+1} \subset \frak g(M')=\frak n(M')= 2 \mathcal{O}_{\frak p}$$ by \cite[Lemma 7.1]{riehm_integral_1964} and Proposition \ref{4.2} and
$${L'}^{\sharp}\cong \pi_\frak p^{-1} A(2, 2\rho_\frak p) \ \ \ \text{with} \ \ \ \frak n({L'}^\sharp)=\frak p^{e-1}\supset 2 \mathcal{O}_{\frak p}=\frak n(M') , $$ one concludes that $M'$ represents $L'$ by \cite[Theorem 7.2]{riehm_integral_1964} and Theorem \ref{2.1}.

\medskip

\underline{Necessity}. Suppose $M$ is a quaternary $2$-universal integral  lattice with a Jordan splitting $$M=M_1\perp \cdots \perp M_t . $$
Since $\frak n(M)\subseteq \mathcal{O}_{\frak p}$, one has $\frak s(M)=\frak s(M_1) \subseteq 2^{-1}\mathcal{O}_{\frak p}$. On the other hand, $M$ represents $2^{-1}A(0, 0)$. It follows that $$\frak s(M)=\frak s(M_1) =2^{-1}\mathcal{O}_{\frak p} \ \ \ \text{with} \ \  \ \rank (M_1)\geq 2 . $$ Therefore $\frak n(M_1)=2\frak s(M_1)$ and $\rank (M_1)=2$ or $4$. Since $F_\frak p M\cong \Bbb H\perp \Bbb H$ by Theorem \ref{2.1}, one concludes that $$M\cong 2^{-1} A(0, 0) \perp 2^{-1} A(0, 0)$$ for $M=M_1$ by \cite[93:18.Example (vi)]{omeara_quadratic_1963}.

Now we only need to consider the case $\rank (M_1)=2$. Since $M$ represents $2^{-1}A(0, 0)$, we can assume that $M_1 \cong 2^{-1}A(0, 0)$ in a basis  $\{x_1, x_2\}$ by \cite[82:15]{omeara_quadratic_1963}. Since $M$ also represents $2^{-1} A(2, 2\rho_\frak p)$, we have $$ \frak n(M_2\perp \cdots \perp M_t) = \mathcal{O}_{\frak p} $$ by Lemma \ref{4.4}.
Since $F_\frak p(M_2\perp \cdots \perp M_t) \cong \Bbb H$, one can write
\begin{equation} \label{complement} M_2\perp \cdots \perp M_t = \begin{cases} \pi_\frak p^{-i} A(\epsilon \pi_\frak p^i, 0) \ \ \ & \text{for $t=2$} \\
\langle \epsilon \rangle \perp \langle \eta \pi_\frak p^r\rangle \ \ \ & \text{for $t=3$} \end{cases}  \end{equation}
in some basis $\{ z_1, z_2\} $, where $\epsilon$ and $\eta$ are in $\mathcal{O}_{\frak p}^\times$, $r\geq 2$ and $0\leq i\leq e-1$. Let  $\{y_1, y_2\}$ be a basis for $ 2^{-1} \pi_\frak p A(2, 2\rho_\frak p) $. Since $2^{-1}\pi_\frak p A(2, 2\rho_\frak p) \longrightarrow M$, there are elements $a, b, c, d$ in  $\mathcal{O}_{\frak p}$ such that
$$  y_1= a x_1+b x_2 + cz_1+dz_2 .$$
We claim that $a$ or $b$ is in  $\mathcal{O}_{\frak p}^\times$.  Suppose that both $a$ and $b$ lie in $\frak p$. Then
\begin{equation} \label{val} \pi_\frak p=ab+ \begin{cases}  \epsilon c^2+ 2\pi_\frak p^{-i} cd \ \ \ & \text{for $t=2$} \\
 \epsilon c^2 + \eta \pi_\frak p^r d^2 \ \ \ & \text{for $t=3$.} \end{cases} \end{equation}
This implies $c\in \frak p$, and by comparing $\frak p$-adic valuations we can find a contradiction to (\ref{val}).

Without loss of generality, we assume $a\in \mathcal{O}_{\frak p}^\times$.  Then $M=(\mathcal{O}_{\frak p} y_1 +\mathcal{O}_{\frak p} x_2 )\perp N' $ with $$(\mathcal{O}_{\frak p} y_1 +\mathcal{O}_{\frak p} x_2 )\cong 2^{-1} A(0, 0) \ \ \ \text{and} \ \ \ N'\cong M_2\perp \cdots \perp M_t$$ by \cite[82:15, 93:11.Example and 93:14.Theorem]{omeara_quadratic_1963}. Let $\{z_1', z_2'\}$ be the basis of $N'$ corresponding to the above isometry.
Since $ 2^{-1}\pi_\frak p A(2, 2\rho_\frak p) \longrightarrow M$,  there are elements $\alpha, \beta, \gamma, \delta$ in  $\mathcal{O}_{\frak p}$ such that
$y_2= \alpha y_1+\beta x_2 + \gamma z_1'+\delta z_2'$. Then
$$ B(y_1, y_2)= \alpha Q(y_1) + \beta B(y_1, x_2) + \gamma B(y_1, z_1')+\delta B(y_1, z_2') \in 2^{-1} \frak p . $$
Since $\frak s(N') \subseteq 2^{-1} \frak p$, both $B(y_1, z_1')$ and $B(y_1, z_2')$ are in $2^{-1} \frak p$. Therefore $\beta B(y_1, x_2)\in 2^{-1} \frak p$. Since $B(y_1, x_2)\mathcal{O}_{\frak p}=2^{-1} \mathcal{O}_{\frak p}$, one has $\beta\in \frak p$. This implies that $$2^{-1}\pi_\frak p A(2, 2\rho_\frak p) \longrightarrow (\mathcal{O}_{\frak p} y_1+\frak p x_2 )\perp N' . $$
Since $\frak n(N') = \mathcal{O}_{\frak p}$, one obtains
$$2^{-1}\pi_\frak p A(2, 2\rho_\frak p) \longrightarrow (\mathcal{O}_{\frak p} y_1+\frak p x_2 )\perp (\frak p z_1'+\mathcal{O}_{\frak p} z_2') $$ by inspecting the representation of $y_2$. Moreover one has $ \frak n (\frak p z_1'+\mathcal{O}_{\frak p} z_2') = \frak p^2 $ by (\ref{complement}).  A contradiction is derived by applying Lemma \ref{4.4} with scaling $\pi_\frak p^{-1}$.
\end{proof}

Since $2$ is not invertible in $\mathcal{O}_{\mathfrak{p}}$ in the dyadic case, the classic integral lattices are quite different from the usual ones.

\begin{prop} \label{4.6} There is no classic $2$-universal quaternary lattice over any dyadic local field.
\end{prop}

\begin{proof}  Suppose that $M$ is a classic 2-universal quaternary lattice with a Jordan splitting $$M=M_1\perp \cdots \perp M_t . $$ Since $M$ represents all binary unimodular lattices, both proper and improper binary unimodular lattices split $M$ by \cite[82:15]{omeara_quadratic_1963}. This implies that $\rank (M_1)\geq 3$ by \cite[91:9.Theorem]{omeara_quadratic_1963}. Since $M$ represents $A(0,0)$, one can assume that $A(0,0)$ splits $M_1$ by \cite[82:15]{omeara_quadratic_1963}.

\medskip

If $\rank (M_1)= 3$, one can write $$M_1=(\mathcal{O}_{\frak p} x+\mathcal{O}_{\frak p}y)\perp \mathcal{O}_{\frak p} z $$ with $Q(x)=Q(y)=0$, $B(x,y)=1$ and $Q(z)=\epsilon \in \mathcal{O}_{\frak p}^\times$. Since $M$ is classic 2-universal, one obtains $F_\frak p M\cong \Bbb H \perp \Bbb H$ from Theorem \ref{2.1}. This implies that $\frak s (M_2) \subseteq \frak p^2$. Moreover, by the condition $ \langle \epsilon(1+\pi_\frak p)\rangle \longrightarrow M $, there are elements $a, b, c\in \mathcal{O}_{\frak p}$ and $w\in M_2$ such that
$$ \epsilon(1+\pi_\frak p)=Q(ax+by+cz+w)=2ab + \epsilon c^2+ Q(w) .$$ Since $Q(w)\in \frak p^2$, one concludes that $c\in \mathcal{O}_{\frak p}^\times$ and
$$ \frak d(1+\pi_\frak p)=\frak d(c^2+2ab \epsilon^{-1} + Q(w)\epsilon^{-1}) \subseteq \frak p^2 $$ if $e>1$. But this contradicts  \cite[63:5]{omeara_quadratic_1963}.
For $e=1$, we consider a basis  $\{u, v\}$ of  $A(2, 2\rho_\frak p)$. By the representability of this lattice by $M$, one can find elements $\alpha_1, \beta_1, \gamma_1$ and $\alpha_2, \beta_2, \gamma_2$ in  $\mathcal{O}_{\frak p}$ such that
$$ \begin{cases} u= \alpha_1x+\beta_1 y + \gamma_1 z+ w_1 \\
v=\alpha_2 x+\beta_2 y +\gamma_2 z+ w_2 \end{cases} $$ for some $w_1, w_2\in M_2$. Therefore
$$ 2= 2\alpha_1\beta_1+ \gamma_1^2 \epsilon +Q(w_1) \ \ \ \text{and} \ \ \ 2\rho_\frak p= 2\alpha_2\beta_2 + \epsilon \gamma_2^2 +Q(w_2) .$$
This implies that $\gamma_1, \gamma_2 \in \frak p$. One concludes that $$ A(2, 2\rho_\frak p) \longrightarrow (\mathcal{O}_{\frak p} x+\mathcal{O}_{\frak p}y)\perp \frak p z
\perp M_2 . $$
A contradiction is derived by applying Lemma \ref{4.4} with scaling 2.

\medskip

Otherwise $M=M_1$. Since $ A(1,0) \longrightarrow M $ and $A(\pi_\frak p, 0) \longrightarrow M $, the weight  $\frak w(M)$ of $M$ is equal to $\frak p$ by \cite[82:15 and 93:5.Example]{omeara_quadratic_1963}. Therefore $$M\cong A(1,0)\perp A(\pi_\frak p, 0) \ \ \ \text{in some basis} \ \ \ \{w, x, y, z\} $$ by \cite[93:18.Example (vi)]{omeara_quadratic_1963}. Since the lattice $A(2, 2\rho_\frak p)$ with basis $\{u, v\}$ is represented by $M$, there exist $\alpha_1, \beta_1, \gamma_1, \delta_1$ and $\alpha_2, \beta_2, \gamma_2, \delta_2$ in $\mathcal{O}_{\frak p}$ such that
$$ \begin{cases} u= \alpha_1 w+\beta_1 x + \gamma_1 y+ \delta_1 z \\
v=\alpha_2 w+ \beta_2 x+\gamma_2 y +\delta_2 z  .\end{cases} $$
This implies that
$$ 2=\alpha_1^2 + 2\alpha_1\beta_1+\pi_\frak p \gamma_1^2+ 2\gamma_1\delta_1 \ \ \ \text{and} \ \ \ 2\rho_\frak p=\alpha_2^2 +2\alpha_2\beta_2 + \pi_\frak p \gamma_2^2 + 2\gamma_2\delta_2 $$ and one obtains that $\alpha_1, \alpha_2$ are both in $\frak p$.

For $e>1$, one further has $ \gamma_1, \gamma_2\in \frak p$. Therefore
$$1= B(u,v)=\alpha_1\beta_2+\beta_1\alpha_2+\gamma_1\delta_2+\delta_1\gamma_2 \in \frak p $$ which is a contradiction.

For $e=1$, one has $A(\pi_\frak p, 0)\cong A(0,0)$ by \cite[93:11.Example]{omeara_quadratic_1963}. Then
$$A(2, 2\rho_\frak p) \longrightarrow (\frak p w+ \mathcal{O}_{\frak p} x) \perp (\mathcal{O}_{\frak p} y +\mathcal{O}_{\frak p} z) \cong A(0, 0) \perp \pi_\frak p A(\pi_\frak p, 0) .$$
This contradicts Lemma \ref{4.4} with scaling $2$. \end{proof}

\section{Existence of 2-LNG lattices}\label{sec5}

In next two sections, we apply the results in the previous sections to study existence of (classic) $k$-LNG lattices (Definition\;\ref{1.3} (4)) over a number field.
We keep the same notations as in the previous sections.
By Corollary \ref{2.3}, there may exist $k$-LNG lattices over a number field only for $k=1$ or $2$. In this section, we treat the case $k=2$.
Our main result in this section is the following.

\begin{thm}\label{5.1}
	Let $ F $ be a number field. Then there is a $2$-LNG lattice over $F$ if and only if the class number of $F$ is even. Moreover, all $2$-LNG lattices are in the genus of $2^{-1}A(0, 0)\perp 2^{-1}A(0, 0)$.
\end{thm}
\begin{proof}
	 \underline{Necessity.} Let $M$ be a 2-LNG lattice over $F$. Since $M_\frak p$ is 2-universal for a non-dyadic prime $\frak p$, one concludes that $\rank (M)\geq 4$ by Proposition \ref{3.3}. Suppose that the class number of $F$ is odd. Since $M_\frak p$ is 2-universal for all $\frak p\in \Omega_F$, one has that $M_\frak p$ is also 1-universal. 	
So there is a single class in $\gen(M)$ by \cite[Proposition 3.2]{xu_indefinite_2020}. This implies that $M$ is indefinite 2-universal over $F$. A contradiction is derived.	
	
\medskip	
	
	\underline{Sufficiency.}  Let
$$ M=2^{-1}A(0, 0)\perp 2^{-1}A(0, 0)=(\mathcal O_F x+ \mathcal O_F y) \perp (\mathcal O_F z+ \mathcal O_F w)$$ where  $Q(x)=Q(y)=Q(z)=Q(w)=0$ and $B(x, y)=B(z, w)=\frac{1}{2} $.
Since the class number of $F$ is even, there exists an unramified quadratic extension $ E=F(\sqrt{c}) $ with $ c\in \mathcal{O}_{F} $.   Since $ E=F(\sqrt{c}) $ is an unramified quadratic extension, one concludes that $\ord_\frak p(c) \equiv 0 \pmod 2$ for all $\frak p\in \Omega_F\setminus \infty_F$.

Consider the quadratic space $Fu\perp Fv$ with
$  Q(u)=1$ and $Q(v)=-c$.    By \cite[81:14]{omeara_quadratic_1963}, we can define a binary $\mathcal O_F$-lattice $L$ on $Fu\perp Fv$ by local conditions as follows
$$ L_\frak p= \begin{cases} \mathcal{O}_{\frak p} u \perp \mathcal{O}_{\frak p} \pi_\frak p^{-t_\frak p} v \ \ \ & \text{if $\frak p$ is non-dyadic} \\
2^{-1}A(0, 0) \ \ \ & \text{if $\frak p$ is dyadic and splits completely in $E/F$}\\
2^{-1}A(2, 2\rho_\frak p) \ \ \ & \text{if $\frak p$ is dyadic and inert in $E/F$} \end{cases} $$
where $t_\frak p=\frac{1}{2} \ord_{\frak p} (c)$. We claim that there is an $\mathcal O_F$-lattice in $\gen(M)$ which cannot represent $L$. Then such a lattice is 2-LNG.
By \cite[Theorem 4.1]{hsia_indefinite_1998}, the claim follows as soon as we prove
\begin{equation}\label{spn}
\theta_\frak p (X(M_\frak p/K_\frak p ) ) =N_{\frak P\mid \frak p} (E_\frak P^\times)
\end{equation}
for all $\frak p\in \Omega_F\setminus \infty_F$, where $\frak P$ is a prime of $E$ above $\frak p$, $N_{\frak P\mid \frak p}$ is the norm map from $E_\frak P$ to $F_\frak p$, $K_\frak p$ is a fixed sublattice of $M_\frak p$ satisfying $K_\frak p \cong L_\frak p$ and
$$X(M_\frak p/K_\frak p)= \{ \sigma \in O^+(F_\frak p M_\frak p): \ K_\frak p \subset \sigma M_\frak p \} 	.$$
	
When $\frak p$ is a non-dyadic prime of $F$,  \eqref{spn} holds by \cite[92:5]{omeara_quadratic_1963} and \cite[Theorem 5.1]{hsia_indefinite_1998}.  	
	
When $\frak p$ is dyadic and splits completely in $E/F$, we have
 $$\theta_\frak p (X(M_\frak p/K_\frak p ) )\supseteq \theta_\frak p(O^+((F_\frak p K_\frak p)^\perp))= \theta_\frak p(O^+(\Bbb H))= F_\frak p^\times $$
 where $(F_\frak p K_\frak p)^\perp$ is the orthogonal complement of $F_\frak p K_\frak p$ in $F_\frak pM_\frak p$. Therefore (\ref{spn}) holds.

When $\frak p$ is dyadic and inert in $E/F$, we have the splitting
$$M_\frak p=K_\frak p\perp K_\frak p^\perp$$
by \cite[82:15]{omeara_quadratic_1963}. Since $\frak n(M_\frak p)=2\frak s(M_\frak p)$, 	one obtains that $\frak n(K_\frak p^\perp)=2\frak s(K_\frak p^\perp)$. Therefore $$K_\frak p^\perp \cong 2^{-1}A(2, 2\rho_\frak p)$$ by \cite[93:11.Example]{omeara_quadratic_1963}.  For any $\sigma\in X(M_\frak p/K_\frak p )$, one also has the splitting
$$ \sigma M_\frak p=K_\frak p \perp K_\sigma \ \ \ \text{with} \ \ \ K_\sigma\cong 2^{-1}A(2, 2\rho_\frak p) $$ by the same argument as above. Since both $ K_\frak p^\perp$ and $K_\sigma$ are $\mathcal{O}_{\frak p}$-maximal lattices on $F_\frak p K_\frak p^\perp$ by \cite[82:19]{omeara_quadratic_1963}, we have $K_\frak p^\perp = K_\sigma$ by
\cite[91:1.Theorem]{omeara_quadratic_1963}. This implies that
$$ 	 X(M_\frak p/K_\frak p ) =O^+(M_\frak p) \ \ \ \text{and} \ \ \ \theta_\frak p( X(M_\frak p/K_\frak p )) =\theta (O^+(M_\frak p))= \mathcal{O}_{\frak p}^\times (F_\frak p^\times)^2 $$
by \cite[Lemma 1]{hsia_spinor_1975}. Therefore \eqref{spn} follows from \cite[63:16.Example]{omeara_quadratic_1963}. The proof of the above claim is complete. 	
	
\medskip

Suppose that $N$ is a 2-LNG lattice over $F$. Then  $\rank(N)=4$ by \cite[\S 4  (5)]{hsia_indefinite_1998}. Since $N_\frak p$ is universal for all $\frak p\in \Omega_F$, one concludes that $$N\in \gen(2^{-1}A(0, 0)\perp 2^{-1}A(0, 0))$$ by Theorem \ref{2.1}, Proposition \ref{3.3} and Proposition \ref{4.5}.
\end{proof}

It should be pointed out that there are no classic $2$-LNG lattices over a number field by Proposition\;\ref{4.6}.
Here we provide a concrete example of 2-LNG integral quadratic form.

\begin{ex} Let $F=\Bbb Q(\sqrt{-5})$. Then
$$ (1+\sqrt{-5})x^2+5xy+(1-\sqrt{-5})y^2 + zw $$ is a 2-LNG integral quadratic form over $F$.
	\end{ex}
\begin{proof} It is well known that the class number of $F$ is two and the Hilbert class field of $F$ is $F(\sqrt{-1})$  (see  \cite[Example 3.8]{xu_indefinite_2020}). The quadratic form $xy+zw$ corresponds to the  $\mathcal O_F$-lattice
$$M=2^{-1}A(0, 0)\perp 2^{-1}A(0, 0)= (\mathcal O_F v_1+ \mathcal O_{F} v_2) \perp (\mathcal O_F v_3 + \mathcal O_{F} v_4) $$
where $Q(v_1)=Q(v_2)=Q(v_3)=Q(v_4)=0$ and $B(v_1, v_2)=B(v_3, v_4)=\frac{1}{2}$. Using by \cite[92:5]{omeara_quadratic_1963} and \cite[Lemma 1]{hsia_spinor_1975}, we find that
   $$\theta_\frak p (O^+(M_\frak p))= \mathcal{O}_{\frak p}^\times (F_\frak p^\times)^2 $$ for $\frak p\in \Omega_F\setminus \infty_F$. Thus,  the number $h(M)$  of proper classes in $\gen(M)$ is given by
   $$ h(M)=\bigg[\Bbb I_F : F^\times \bigg(\prod_{\frak p\in \infty_F} F_\frak p^\times\times \prod_{\frak p\in \Omega_F\setminus \infty_F} \mathcal{O}_{\frak p}^\times \bigg) \Bbb I_F^2\bigg] = [\pic(\mathcal O_F) : \pic(\mathcal O_F)^2]=2$$
 by \cite[33:14.Theorem, 102:7 and 104:5.Theorem]{omeara_quadratic_1963}. Here $\Bbb I_F$ denotes the id\`ele group of $F$.

Let $\frak q= (2, 1+\sqrt{-5})$ be the non-principal prime ideal of $F$ above $2$.  Define
$$N = (\frak q v_1 + \frak q^{-1} v_2) \perp (\mathcal O_F v_3 + \mathcal O_{F} v_4)\,,  $$
which corresponds to the integral quadratic form $$(1+\sqrt{-5})x^2+5xy+(1-\sqrt{-5})y^2 + zw $$ as shown in \cite[Example 3.8]{xu_indefinite_2020}. Then $N_\frak p=M_\frak p$ for $\frak p\neq \frak q$ and $\sigma_\frak q (M_{\mathfrak{q}})= N_{\mathfrak{q}}$ where
$$\sigma_\frak q v_1=\pi_\mathfrak{q} v_1, \ \ \ \sigma_{\mathfrak{q}}  v_2= \pi_{\mathfrak{q}}^{-1} v_2, \ \ \ \sigma_\frak q v_3=v_3 \ \ \ \text{and} \ \ \ \sigma_\frak q v_4= v_4 .$$
Since $\sigma_\frak q = \tau_{v_1-v_2} \tau_{v_1-\pi_\frak qv_2}$, one obtains that $\theta_\frak q(\sigma_\frak q)=\pi_\frak q (F_\frak q^\times)^2$. On the other hand, since $2$ is inert in $\Bbb Q(\sqrt{5})/\Bbb Q$, one concludes that the dyadic prime $\frak q$ of $F$ is inert in $F(\sqrt{-1})=F(\sqrt{5})$ by inspecting the extension degree of the residue fields. By using \cite[Chapter VI, \S5. (5.7) Corollary]{Neu1}, we see that
$$(i_\frak p)_{\frak p\in \Omega_F} \not \in F^\times \bigg(\prod_{\frak p\in \infty_F} F_\frak p^\times\times \prod_{\frak p\in \Omega_F\setminus \infty_F} \mathcal{O}_{\frak p}^\times \bigg) \Bbb I_F^2$$ where
$$ i_\frak p= \begin{cases} \pi_\frak q \ \ \ & \text{if } \frak p=\frak q  \\
1 \ \ \ & \text{otherwise.} \end{cases} $$
Therefore  $M$ and $N$ are exactly the representatives of the two classes in $\gen(M)$.

Put
$$ L= \mathcal O_F (v_1+ v_2) + \mathcal O_F (\sqrt{-5} v_2+v_3-v_4) \subset M\,. $$
 Since $\det(FL)\cdot \det(FM)\in (F^\times)^2$,
 $$E=F(\sqrt{-\det(FL)\cdot \det(FM)})=F(\sqrt{-1})$$
 is the Hilbert class field of $F$. Moreover, one has
  $$\theta_\frak p (X(M_\frak p/L_\frak p ) ) =N_{\frak P\mid \frak p} (E_\frak P^\times) $$
 for any prime $\frak p\neq \frak q$ of $F$ by \cite[Theorem 5.1]{hsia_indefinite_1998},  where
$$X(M_\frak p/L_\frak p)= \{ \sigma \in O^+(F_\frak p M_\frak p): \ L_\frak p \subset \sigma M_\frak p \}.$$
Since $ L_\frak q \cong 2^{-1} A(2, 2\rho_\frak q)$, the same arguments as in Theorem\;\ref{5.1} and \cite[63:16.Example]{omeara_quadratic_1963} yield
$$X(M_\frak q/L_\frak q)= \{ \sigma \in O^+(F_\frak q M_\frak q): \ L_\frak q \subset \sigma M_\frak q \}=O^+(M_\frak q) \ \ \ \text{and} \ \ \ \theta_\frak q(X(M_\frak q/L_\frak q))=N_{\frak Q\mid \frak q} (E_\frak Q^\times)\,, $$
  $\frak Q$ denoting a prime of $E$ above $\frak q$. One concludes that $L$ is not represented by $N$ by \cite[Theorem 4.1]{hsia_indefinite_1998}. Hence $N$ is  $2$-LNG over $F$.
\end{proof}

\section{Existence of classic 1-LNG lattices}\label{sec6}

In \cite{xu_indefinite_2020}, it has been proved that a number field $F$ admits 1-LNG lattices if and only if the class number of $F$ is even.  However, the lattices constructed in \cite[Theorem 3.5.(2)]{xu_indefinite_2020} are not classic integral. It is natural to ask what kind of number fields admit classic 1-LNG (see \cite[Remark 1.2]{xu_indefinite_2020}).

Recall that a prime $\frak p\in\infty_F$ is called ramified in a finite extension $E/F$ if $\frak p$ is real and $E$ has a complex prime above $\frak p$ (\cite[p.172]{Janusz}). We say a finite extension $E/F$ is  unramified if it is unramified at all (finite or infinite) primes of $F$.

\begin{thm} \label{6.1} A number field $F$ admits a classic 1-LNG $\mathcal O_F$-lattice if and only if there is a quadratic unramified extension  $E/ F $ in which every dyadic prime of $F$ splits completely. In this case, there are infinitely many classes of 1-LNG free $\cO_F$-lattices.
\end{thm}
\begin{proof} \underline{Necessity}. Let $M$ be a classic 1-LNG $\mathcal{O}_F$-lattice and $K$ be an integral $\cO_F$-lattice of rank 1 such that $K$ is not represented by $M$.
Since there is no classic 1-universal binary lattice over dyadic local fields by \cite[Corollary 2.9]{xu_indefinite_2020}, one obtains $\rank (M)\geq 3$.
By \cite[p.135, line 2-4]{hsia_indefinite_1998}, one concludes that $$\rank (M)=3 \ \ \ \text{ and } -\det(FK) \cdot \det(FM) \not \in (F^\times)^2 . $$ Let $ E=F(\sqrt{-\det(FK)\cdot \det(FM)}) $ and $L\in \gen(M)$ such that $K\subset L$. Then
$$ \theta_\frak p (X(L_\frak p, K_\frak p) ) = N_{\mathfrak{P}\mid \mathfrak{p}}(E_\frak P^\times)$$
by \cite[Theorem 4.1]{hsia_indefinite_1998}, where $\frak P$ is a prime in $E$ above $\frak p$ and $N_{\frak P\mid \frak p}$ is the norm map from $E_\frak P$ to $F_\frak p$ and
$$ X(L_\frak p, K_\frak p) = \{ \sigma\in O^+(F_\frak p M_\frak p): \ K_\frak p\subset \sigma (L_\frak p) \} $$ for all $\frak p\in \Omega_F$.
Since $L_\frak p$ is 1-universal over $F_\frak p$ for all $\frak p\in \Omega_F$, one obtains
$$\theta_\frak p (X(L_\frak p, K_\frak p) ) \supseteq  \theta_\frak p(O^+(L_\frak p)) \supseteq  \begin{cases}  \mathcal{O}_{\frak p}^\times \ \ \ & \frak p\in \Omega_F\setminus \infty_F \\
 F_\frak p^\times \ \ \ & \frak p\in \infty_F \end{cases} $$
by \cite[Lemma 2.2]{xu_indefinite_2020}. So $N_{\mathfrak{P}\mid \mathfrak{p}}(E_\frak P^\times)\supseteq \cO_{\frak p}^\times$ for all $\frak p\in\Omega_F$, and hence $E/F$ is an unramified extension by \cite[Chap.\;VI, (6.6) Corollary]{Neu1}.

 Now let $\frak p$ be a dyadic prime of $F$. Since $L_\frak p$ is classic integral,  the symmetry
 $\tau_v$  is an element in $O(L_\frak p)$ for any $v\in L_\frak p$ with $0\leq \ord_\frak p(Q(v))\leq 1$.  Then
 $ \theta_\frak p(O^+(L_\frak p)) = F_\frak p^\times $
 by  \cite[Lemma 2.2]{xu_indefinite_2020}.
 Therefore $N_{\mathfrak{P}\mid \mathfrak{p}}(E_\frak P^\times)=F_{\frak p}^\times$, and hence $E_{\mathfrak{P}}=F_{\frak p}$ by local class field theory (cf. \cite[Chap.\;V,\,(1.3) Theorem]{Neu1}). This proves that every dyadic prime of $F$ splits completely in $E/F$ as desired.

\

\underline{Sufficiency}. Suppose that $E/F$ is a quadratic unramified extension in which dyadic primes of $F$ all split completely. We can write $E=F(\sqrt{-b})$ for some  $b\in \mathcal{O}_{F} $. Since $E/F$ is unramified, one has that $ \ord_{\frak p}(b) $ is even for all $ \mathfrak{p}\in \Omega_{F}\backslash \infty_{F} $. Consider the ternary space $V=[1,\,-1,\,-b]$. This is the unique isotropic ternary space with  $\det(V)=b$.

Consider an $\cO_F$-lattice $L=\langle 1,\,-1,\,-b\rangle$ on $V$.
Write $\beta_{\mathfrak{p}}:=b\pi_{\mathfrak{p}}^{-\ord_{\frak p}(b)} $ for each $\frak p\in \Omega_F\setminus\infty_F$. Then $ \beta_{\frak p}\in \mathcal{O}_{F_{\mathfrak{p}}}^{\times } $. Since $\ord_{\frak p}(b)=0$ for almost all $\frak p$, one has $L_{\frak p}=\langle 1, -1, -\beta_{\frak p}\rangle$ at almost all primes. When $\frak p$ is dyadic, set $\kappa_{\frak p}:=1+4\rho_{\mathfrak{p}}\pi_{\mathfrak{p}}^{-1} $. Then $(\kappa_{\frak p}, 1+\pi_{\mathfrak{p}})_{\mathfrak{p}}=-1 $ by Lemma \ref{4.1}.

By \cite[81:14]{omeara_quadratic_1963}, one can define a ternary $ \mathcal{O}_{F} $-lattice $ M $ on $ V $ by local conditions as follows:
\begin{align*}
	M_{\mathfrak{p}}=
	\begin{cases}
		\langle 1, -1, -\beta_{\frak p}  \rangle  &\text{if $ \mathfrak{p} $ is non-dyadic},\\
		A(1,-\pi_{\mathfrak{p}})\perp \langle -(1+\pi_{\mathfrak{p}})\beta_{\frak p} \rangle &\text{if $ \mathfrak{p} $ is dyadic and $ (1+\pi_{\mathfrak{p}},-\beta_{\mathfrak{p}})_{\mathfrak{p}}=1 $},\\
		A(\kappa_{\mathfrak{p}},-\kappa^{-1}_{\mathfrak{p}}\pi_{\mathfrak{p}})\perp \langle -(1+\pi_{\mathfrak{p}})\beta_{\frak p} \rangle &\text{if $ \mathfrak{p} $ is dyadic and $ (1+\pi_{\mathfrak{p}},-\beta_{\mathfrak{p}} )_{\mathfrak{p}}=-1$}\,.
	\end{cases}
\end{align*}

When $ \mathfrak{p} $ is a non-dyadic prime, $ M_{\mathfrak{p}} $ is unimodular and so is 1-universal by \cite[Proposition 2.3]{xu_indefinite_2020}. When $ \mathfrak{p} $ is a dyadic prime, using Lemma \ref{4.1} we can easily check that  $ S_{\mathfrak{p}}(F_{\frak p}M)=(-1,-1)_{\mathfrak{p}}$ and so $ F_{\mathfrak{p}}M  $ is isotropic by \cite[58:6]{omeara_quadratic_1963}. Also, we have the weight $ \mathfrak{w}(M_{\mathfrak{p}})=\mathfrak{p} $ by \cite[93:5]{omeara_quadratic_1963}. So $ M_{\mathfrak{p}} $ is 1-universal by \cite[Proposition 2.17]{xu_indefinite_2020}.

Fix $x_\frak p \in M_\frak p$ with $Q(x_\frak p)=1$ and write
$$ X(M_\frak p, 1) = \{ \sigma\in O^+(F_\frak p M_\frak p): \ x_\frak p\in \sigma (M_\frak p) \} $$ for all $\frak p\in \Omega_F$. When $ \mathfrak{p} $ is a non-dyadic prime,  $\theta_\frak p (X(M_\frak p, 1))=N_{\frak P\mid \frak p} (E_\frak P^\times) $ by the second part of \cite[Satz 3(a)]{schulzepillot_Darstellung_1980} with $ r=s=0 $. When $\mathfrak{p}$ is a dyadic prime,  $ \theta_\frak p (X(M_\frak p, 1))= F_\frak p^\times $ by \cite[Theorem 2.1 (Case III)]{Xu4}. On the other hand, the prime $\frak p$ splits completely in $E/F$, thus $N_{\mathfrak{P}\mid \mathfrak{p}}(E_\frak P^\times)=F_{\mathfrak{p}}^{\times}$.
We conclude that $\theta_\frak p (X(M_\frak p, 1))=N_{\mathfrak{P}\mid \mathfrak{p}}(E_\frak P^\times) $ for all $\frak p\in \Omega_F$. By \cite[Theorem 4.1]{hsia_indefinite_1998},  there is an $\mathcal O_F$-lattice in $\gen(M)$ which does not represent $1$. This lattice is classic 1-LNG as desired.

Fix $P\in \gen(M)$ such that 1 is not represented by $P$. There is a base $\{z_1, z_2, z_3\}$ for $FP$ such that $P= \frak b_1 z_1+ \frak b_2  z_2+ \frak b_3 z_3$ where $\frak b_1, \frak b_2$ and $\frak b_3$ are fractional ideals of $\cO_F$ by \cite[81:3. Theorem]{omeara_quadratic_1963}. The Steinitz class of $P$ is defined as the class of $\frak b_1\frak b_2\frak b_3$ in $\pic(\cO_F)$ and is independent of choices of bases of $FP$ (cf. \cite[9.3.10, p.\hskip 0.1cm 143]{Voight}). Moreover, $P$ is free if and only if the Steinitz class of $P$ is trivial in $\pic(\cO_F)$.

 By Tchebotarev's density theorem \cite[Chapter V, (6.4) Theorem]{Neu86}, there are infinitely many non-dyadic primes $\frak q_1, \cdots, \frak q_n , \cdots$ such that the class of each $\frak q_i$ in $\pic(\cO_F)$ is equal to the inverse of the Steinitz class of $P$ for $i=1, \cdots, n, \cdots$.  Define a sub-lattice $P_i$ of $P$ by
$$ (P_i)_{\frak p} = \begin{cases} P_\frak p \ \ \ & \text{when $\frak p \neq \frak q_i$} \\
\langle 1, -1, -\beta_{\frak p} \pi_\frak p^2  \rangle \ \ \ & \text{when $\frak p=\frak q_i$} \end{cases} $$
for each $i=1, \cdots, n , \cdots $.  Since the Steinitz class of $P_i$ is equal to the product of the Steinitz class of $P$ and the class of $\frak q_i$ by \cite[82:11]{omeara_quadratic_1963}, one has the Steinitz class of $P_i$ is trivial and $P_i$ is a free $\cO_F$-module for $i=1, \cdots, n, \cdots$. Since $(P_i)_{\frak q_i}$ is also $1$-universal by \cite[Proposition 2.3]{xu_indefinite_2020}, one obtains that $P_i$ is 1-universal over $\cO_\frak p$ for all $\frak p\in \Omega_F$. Since $P$ cannot represent $1$, one concludes that $P_i$ cannot represent $1$ either for $i=1, \cdots, n\cdots$.
 Namely, all $P_i$ are  classic 1-LNG free lattices for $i=1, \cdots, n\cdots$.  Since $\det(P_i)\neq \det(P_j)$ for $i\neq j$, none of them are in the same class as desired.
\end{proof}

In the rest of this section, we determine all quadratic fields over which there are classic 1-LNG lattices. The starting point of the arguments is that all quadratic unramified  extensions of a quadratic field can be described explicitly by the genus theory.

\begin{prop}\label{6.3} Let $F$ be a quadratic field with discriminant $d_F$ and let $p_1, \cdots, p_t$ be all the odd prime divisors of $d_F$.  Write $p_i^*=(-1)^{\frac{p_i-1}{2}}p_i$ for each $1\leq i\leq t$ and put
\[
G^{(+)}=F(\sqrt{p_1^*},\cdots, \sqrt{p_t^*})=\Q(\sqrt{d_F}\,,\,\sqrt{p_1^*},\cdots, \sqrt{p_t^*})\,.
\]  Here $G^{(+)}=F$ if $t=0$.
Then a quadratic extension $E/F$ is unramified if and only if $E\subseteq G^{(+)}$ and in case $F$ is real, $E$ is totally real.
\end{prop}
\begin{proof}
Let $H$ be the Hilbert class field of $F$. Then the Artin map gives an isomorphism
between $ \pic(\mathcal O_F) $ and $ \gal(H/F) $ by \cite[Chap.\,VI, (6.9) Proposition]{Neu1}.
  The subgroup $\pic(\mathcal O_F)^2$ generated by $\{ \frak a^2: \frak a\in \pic(\mathcal O_F)\} $
  corresponds to an intermediate field $G$ of $H/F$. It is the \emph{genus field} of $F$ in the sense of \cite[\S\,VI.3, pp.243--244]{Janusz}, characterized as the maximal unramified abelian extension of $F$ which is abelian over $\Q$. Note that if $F$ is real, then $G$ must be totally real for otherwise it would be ramified at a real prime of $F$.

The field $G^{(+)}$ is called the \emph{extended genus field} of $F$ in  \cite[\S\,VI.3]{Janusz} (but in \cite{Ishida} it is called the \emph{genus field} of $F$). It is the maximal abelian extension of $\Q$ containing $F$ which is unramified at all finite primes of $F$ (\cite[pp.3--4]{Ishida}). So clearly $G\subseteq G^{(+)}$. If $F$ is imaginary, then $G=G^{(+)}$. If $F$ is real, then $G$ coincides with the maximal totally real subfield of $G^{(+)}$.

If $E\subseteq G$, then as a subextension of an unramified extension, $E/F$ is unramified. Conversely, assume that $E/F$ is unramified. Then $E\subseteq H$. Since the quotient group
$$\mathrm{Gal}(H/F)/\mathrm{Gal}(H/E)\cong\mathrm{Gal}(E/F) $$ is 2-torsion, we have $\mathrm{Gal}(H/G)\subseteq \mathrm{Gal}(H/E)$. Hence $E\subseteq G$. This completes the proof.
\end{proof}


\begin{thm} \label{6.4} Let $F$ be a quadratic field with discriminant $d_F$ and let $p_1, \cdots, p_t$ be all the odd prime divisors of $d_F$.
Then there is no classic ternary 1-LNG lattice over $\cO_F$ if and only if $F$ is one of the fields in the following table:

\begin{center}
\renewcommand{\arraystretch}{1.5}
\begin{tabular*}{16.3cm}{c|c|c}
\hline
$d_F\pmod{4}$ &  $F$ \text{ real} &  $F$ \text{ imaginary} \\
\hline
\multirow{5}{*}{$\equiv 1$}  & $\mathbb{Q}(\sqrt{p_1})$ \text{ with } $p_1\equiv 1\pmod{4}$  & $\mathbb{Q}(\sqrt{-p_1})$ \text{ with } $p_1\equiv 3\pmod{4}$  \\
\cline{2-3}
 & $\mathbb{Q}(\sqrt{p_1p_2})$ \text{ with } $p_1\equiv p_2\equiv 3\pmod{4}$  & $\mathbb{Q}(\sqrt{-p_1p_2})$ \text{ with } \\
 & \qquad \text{ or } $p_1\equiv p_2\equiv -3 \pmod{8}$ &  $p_1\equiv -p_2\equiv\pm 3\pmod{8}$ \\
 \cline{2-2}
 & $\mathbb{Q}(\sqrt{p_1p_2p_3})$ \text{ with } $p_1\equiv -1\pmod{8}$ & \\
 &  $p_2\equiv 3,\,p_3\equiv -3 \pmod{8}$  & \\
\hline
\multirow{7}{*}{$\equiv 0$} &    $\mathbb{Q}(\sqrt{2})$ &  $\Q(\sqrt{-1}); \ \Q(\sqrt{-2})$\\
\cline{2-3}
 & $\mathbb{Q}(\sqrt{p_1})$ \text{ with } $p_1\equiv 3\pmod{4}$  & $\mathbb{Q}(\sqrt{-p_1})$ \text{ with } $p_1\equiv -3\pmod{8}$ \\
 \cline{2-3}
 & $\mathbb{Q}(\sqrt{2p_1})$ \text{ with } $p_1\equiv -1,\,\pm 3\pmod{8}$  & $\mathbb{Q}(\sqrt{-2p_1})$ \text{ with } $p_1\equiv \pm 3\pmod{8}$ \\
\cline{2-2}
 & $\mathbb{Q}(\sqrt{p_1p_2})$ \text{ or } $\mathbb{Q}(\sqrt{2p_1p_2})$ \text{ with }  &  \\
 & $p_1\equiv -3\pmod{8}$, $p_2\equiv 3\pmod{4}$  & \\
  \cline{2-2}
   & $\mathbb{Q}(\sqrt{2p_1p_2})$ \text{ with } & \\
   & $p_1\equiv -3p_2\equiv -1,\,3\pmod{8}$ & \\
\hline
\end{tabular*}
\end{center}
\end{thm}
\begin{proof}
Let $G^{(+)}=\Q(\sqrt{d_F},\,\sqrt{p_1^*},\cdots, \sqrt{p_t^*})$ as in Proposition\;\ref{6.3}.
If $t=0$, then $G^{(+)}=F$. Thus, $F$ has no quadratic unramified extensions at all, hence there is no classic 1-LNG $\mathcal{O}_F$-lattice by Theorem\;\ref{6.1}. Note that in this case $F=\Q(\sqrt{2}),  \Q(\sqrt{-1})$ or $\Q(\sqrt{-2})$.

Now let us assume $t\ge 1$. We discuss 4 cases to finish the proof.

\noindent {\bf Case 1}.  $t=1$.

If $d_F\equiv 1\pmod{4}$, then $F=\Q(\sqrt{p_1})$ with $p_1\equiv 1\pmod{4}$ or $F=\Q(\sqrt{-p_1})$ with $p_1\equiv 3\pmod{4}$.
In both cases, $d_F=p_1^*$ so that $G^{(+)}=F$. Thus, as in the case $t=0$, there is no classic 1-LNG $\mathcal{O}_F$-lattice.

Next assume $d_F\equiv 0\pmod{4}$. If $F$ is imaginary, then $F=\Bbb Q(\sqrt{-p_1})$ with $p_1\equiv 1 \pmod 4$ or $F=\Bbb Q(\sqrt{-2p_1})$. In the former case,  $G^{(+)}=F(\sqrt{p_1})$ is then only quadratic unramified extension of $F$. When $p_1\equiv 1 \pmod 8$, the dyadic prime of $F$ splits completely in $F(\sqrt{p_1})/F$ by \cite[63:1.Local Square Theorem]{omeara_quadratic_1963}. In this case, there is a classic ternary 1-LNG $\mathcal O_F$-lattice by Theorem\;\ref{6.1}.
 If $p_1\equiv 5 \pmod 8$, then $2$ is inert in $\Bbb Q(\sqrt{p_1})/\Bbb Q$. This implies that the dyadic prime of $F$ is inert in $F(\sqrt{p_1})/F$ by inspecting extension degree of the residue fields. In this case, one concludes using Theorem\;\ref{6.1} that there is no classic ternary 1-LNG $\mathcal O_F$-lattice.

For $F=\Bbb Q(\sqrt{-2p_1})$, we have $G^{(+)}=F(\sqrt{p_1^*})/F$. When $p_1\equiv \pm 1 \pmod 8$, the dyadic prime of $F$ splits completely in $F(\sqrt{p_1^*})/F$ by \cite[63:1.Local Square Theorem]{omeara_quadratic_1963}. In this case, there is a classic ternary 1-LNG $\mathcal O_F$-lattice by Theorem\;\ref{6.1}. When $p_1\equiv \pm 3 \pmod 8$, the dyadic prime of $F$ is inert in $F(\sqrt{p_1^*})/F$, so that $F$ has no  classic ternary 1-LNG lattice by Theorem\;\ref{6.1}.

Now suppose $F$ is real. Then $F=\Q(\sqrt{p_1})$ with $p_1\equiv 3\pmod{4}$ or $F=\Q(\sqrt{2p_1})$. If $F=\Q(\sqrt{p_1})$, then $G^{(+)}=F(\sqrt{p_1^*})=F(\sqrt{-p_1})$ is not totally real. It follows that $F$ has no quadratic unramified extension, hence there is no classic ternary 1-LNG lattice over $\cO_F$. If $F=\Q(\sqrt{2p_1})$ with $p_1\equiv 3\pmod{4}$, similar arguments show that $F$ admits no classic 1-LNG lattice. If $F=\Q(\sqrt{2p_1})$ with $p_1\equiv -3\pmod{8}$, then the dyadic prime of $F$ is inert in $G^{(+)}=F(\sqrt{p_1})$, so again there is no classic ternary 1-LNG lattice over $\mathcal O_F$. Finally, if $F=\Q(\sqrt{2p_1})$ with $p_1\equiv 1\pmod{8}$, then the quadratic extension $F(\sqrt{p_1^*})=F(\sqrt{p_1})$ satisfies the conditions in Theorem\;\ref{6.1}, whence the existence of classic ternary 1-LNG lattices over $\mathcal{O}_F$.

\medskip

\noindent {\bf Case 2}. $F$ is imaginary and $t\ge 2$.

If $p_{i}\equiv \pm 1 \pmod 8$ for some $1\leq i\leq t$, then $F(\sqrt{p_{i}^*})/F$ is a quadratic extension contained in $G^{(+)}$, hence unramified. Moreover, by  the local square theorem all dyadic primes of $F$ split completely in $F(\sqrt{p_{i}^*})/F$. By Theorem\;\ref{6.1}, there exist classic ternary 1-LNG lattices over $\cO_F$ in this case.

If there are two distinct indices $1\leq i\neq j \leq t$ such that $p_{i}\equiv p_{j}\mod 8$, then $p_i^*p_j^*=p_ip_j$ and the field $F(\sqrt{p_i^*p_j^*})=F(\sqrt{p_{i} p_{j}})$
 is an unramified quadratic extension by Proposition \ref{6.3}. Since $p_ip_j\equiv 1\pmod{8}$, all dyadic primes of $F$ split completely in  $F(\sqrt{p_ip_j})$. Thus, as in the previous case there exist classic ternary 1-LNG lattices over $\cO_F$.

In the remaining case, we must have $t=2$ and we may assume without loss of generality that $p_1\equiv 3\pmod{8}$ and $p_2\equiv -3\pmod{8}$. Hence $G^{(+)}=F(\sqrt{p_1^*},\,\sqrt{p_2^*})=F(\sqrt{-p_1},\,\sqrt{p_2})$. Note that $F=\Q(\sqrt{-p_1p_2})$ or $F=\Q(\sqrt{-2p_1p_2})$. In the former case, $2$ splits completely in $F$ and is inert in $\Q(\sqrt{p_2})$. This implies that the dyadic primes of $F$ are inert in $G^{(+)}=F(\sqrt{p_2})$. So by Theorem\;\ref{6.1}, there is no classic ternary 1-LNG lattices over $\cO_F$. In the latter case, $G^{(+)}$ contains the quadratic extension $F(\sqrt{-p_1p_2})$ of $F=\Q(\sqrt{-2p_1p_2})$, and the dyadic prime of $F$ splits completely in $F(\sqrt{-p_1p_2})$. So by Theorem\;\ref{6.1}, $F$ admits a classic ternary 1-LNG lattice.

\medskip

\noindent {\bf Case 3}. $F$ is real and $t=2$.

First assume $d_F\equiv 1\pmod{8}$. Then $d_F=p_1p_2$ with $p_1\equiv p_2\pmod{8}$, and  $G^{(+)}=F(\sqrt{p_1^*})=F(\sqrt{p_2^*})$. Similar to the imaginary case, if $p_1\equiv p_2\equiv 1\pmod{8}$,  there exist classic ternary 1-LNG lattices over $\mathcal O_F$. Otherwise either $G^{(+)}$ is not totally real, or $p_1\equiv p_2\equiv 5\pmod{8}$. In the first case, $F$ has no quadratic unramified extension; in the second case, the unique quadratic unramified extension of $F$ is $F(\sqrt{p_1})$, but the dyadic primes of $F$ are inert in $F(\sqrt{p_1})$. Therefore, there is no classic ternary 1-LNG lattice over $\mathcal O_F$ in either case.

Next assume $d_F\equiv 5\pmod{8}$. Then $d_F=p_1p_2$ with $p_1\not\equiv p_2\pmod{8}$. If $p_1$ or $p_2$ is congruent to $1$ modulo $8$, then $F$ has a classic ternary 1-LNG lattice. Otherwise, we may assume $p_1\equiv 3\pmod{8},\,p_2\equiv -1\pmod{8}$. Then $G^{(+)}=F(\sqrt{-p_1})$ is not real, hence there is no classic ternary 1-LNG lattice over $\mathcal O_F$.

Finally, consider the case $d_F\equiv 0\pmod{4}$. Then either $F=\Q(\sqrt{p_1p_2})$ with $p_1p_2\equiv 3\pmod{4}$ or $F=\Q(\sqrt{2p_1p_2})$. If $F=\Q(\sqrt{p_1p_2})$, we may assume without loss of generality that $p_1\equiv 1\pmod{4},\,p_2\equiv 3\pmod{4}$. Then $G^{(+)}=F(\sqrt{p_1},\,\sqrt{-p_2})$ contains precisely 3 quadratic extensions of $F$. The only real extension among these three extensions is $F(\sqrt{p_1})$. Arguing similarly as before, we see that $F$ has no classic ternary 1-LNG lattice if and only if $p_1\equiv -3\pmod{8}$.

 Now suppose $F=\Q(\sqrt{2p_1p_2})$. As before, if $p_1$ or $p_2$ is congruent to 1 modulo 8, or if $p_1\equiv p_2\pmod{8}$, then there exist classic ternary 1-LNG $\mathcal O_F$-lattices. Without loss of generality, we may now assume either $p_1\equiv -1\pmod{8},\,p_2\equiv 3\pmod{8}$, or $p_1\equiv 5\pmod{8},\,p_2\equiv 3\pmod{4}$. In the former case, the only real quadratic extension of $F$ in  $G^{(+)}=F(\sqrt{-p_1},\,\sqrt{-p_2})$ is $F(\sqrt{p_1p_2})$. But the dyadic prime of $F$ is inert in $F(\sqrt{p_1p_2})$ since $p_1p_2\equiv -3\pmod{8}$. In the other case, $G^{(+)}=F(\sqrt{p_1},\,\sqrt{-p_2})$ and the only real quadratic extension of $F$ in $G^{(+)}$ is $F(\sqrt{p_1})$, in which the dyadic prime of $F$ is again inert. Therefore, for $F=\Q(\sqrt{2p_1p_2})$, there is no classic ternary 1-LNG lattice over $F$ if and only if (up to permutation) $p_1\equiv -1\pmod{8},\,p_2\equiv 3\pmod{8}$, or $p_1\equiv 5\pmod{8},\,p_2\equiv 3\pmod{4}$.

\medskip

\noindent {\bf Case 4}. $F$ is real and $t\ge 3$.

As in previous discussions, if $p_i\equiv p_j\pmod{8}$ for two distinct indices $i,\,j$, or if some  $p_i$ is congruent to 1 modulo 8, then there exist classic ternary 1-LNG lattices over $\mathcal O_F$. The only remaining possibility is that $t=3$ and up to permutation, $p_1\equiv -1,\,p_2\equiv 3,\,p_3\equiv -3\pmod{8}$. Thus, $G^{(+)}=F(\sqrt{-p_1},\,\sqrt{-p_2},\,\sqrt{p_3})$.

If $F=\Q(\sqrt{p_1p_2p_3})$, then $G^{(+)}$ is  a biquadratic extension of $F$ and the only real quadratic subextension is $F(\sqrt{p_3})$. But the dyadic primes of $F$ are inert in $F(\sqrt{p_3})$. So $F$ has no classic ternary 1-LNG lattice in this case.

If $F=\Q(\sqrt{2p_1p_2p_3})$, then $F(\sqrt{p_1p_2p_3})$ is a  real quadratic extension of $F$ contained in $G^{(+)}$ and the dyadic primes of $F$ split completely in it. Therefore classic ternary 1-LNG $\mathcal O_F$-lattices exist.

\medskip

The theorem follows by summarizing all the results in the above discussions.
\end{proof}

As was shown in \cite[Example 3.8]{xu_indefinite_2020}, there are infinitely many integral (non-classic) ternary 1-LNG forms over the field $\Bbb Q(\sqrt{-5})$. However,
Theorem\;\ref{6.4} tells us in particular that  over  $\Bbb Q(\sqrt{-5})$ there is no classic 1-LNG lattice. On the other hand, there are classic 1-LNG lattices over $\Bbb Q(\sqrt{-14})$ according to Theorem\;\ref{6.4}.

\begin{ex}  Let $F=\Bbb Q(\sqrt{-14})$ and $m$ be an odd integer. Then the classic integral ternary quadratic forms
$$(2+11 \sqrt{-14})x^2+(106+6\sqrt{-14}) xy + (32-17\sqrt{-14}) y^2 + m^2 z^2  $$
represent all integers of $F$ over $\cO_\frak p$ for all $\frak p\in \Omega_F$ but cannot represent $2$ over $\cO_F$.
\end{ex}
\begin{proof} Since $2$ is ramified in $F=\Bbb Q(\sqrt{-14})/\Bbb Q$, there is a unique dyadic prime $\frak d$ in $F$. By Proposition \ref{6.3},  the genus field $G$ of $F$ is $F(\sqrt{-7})=F(\sqrt{2})$.
Since $-7(=1-8)$ is a square in $F_\frak d$, one concludes that $\frak d$ splits completely in $G/F$. By Theorem \ref{6.1}, there are infinitely many classic 1-LNG lattices. We will construct infinitely many classic 1-LNG integral quadratic forms explicitly. Consider
$$ M=(\mathcal O_F x+ \mathcal O_F y) \perp \mathcal O_F z \ \ \ \text{with} \ \ \ Q(x)=0, \ Q(y)=\sqrt{-14} \ \ \ \text{and} \ \ \ B(x, y)=Q(z)=1 $$ with the corresponding quadratic form $2xy+\sqrt{-14} y^2+ z^2$.
Then $M$ is a classic integral unimodular lattice. By \cite[Proposition 2.3 and Proposition 2.17]{xu_indefinite_2020}, one obtains that $M_\frak p$ is 1-universal over $\cO_\frak p$ for all $\frak p\in \Omega_F\setminus \infty_F$.  Therefore
$$ \theta_\frak p(O^+(M_\frak p)) \supseteq  \begin{cases}  \mathcal{O}_{\frak p}^\times ( F_\frak p^\times)^2 \ \ \ & \frak p\in \Omega_F\setminus \infty_F \\
 F_\frak p^\times \ \ \ & \frak p\in \infty_F  \end{cases} $$
by \cite[Lemma 2.2]{xu_indefinite_2020}.  By \cite[33:14.Theorem, 102:7 and 104:5.Theorem]{omeara_quadratic_1963}, the number $h(M)$ of (proper) classes in $\gen(M)$ is bounded by
   \begin{equation} \label{cls} h(M) \leq \bigg[\Bbb I_F : F^\times \bigg(\prod_{\frak p\in \infty_F} F_\frak p^\times\times \prod_{\frak p\in \Omega_F\setminus \infty_F} \mathcal{O}_{\frak p}^\times \bigg) \Bbb I_F^2\bigg] \leq [\pic(\mathcal O_F) : \pic(\mathcal O_F)^2]=[G:F]=2 \end{equation}
 where $\Bbb I_F$ denotes the id\`ele group of $F$. On the other hand, write
$$ X(M_\frak p, 2) = \{ \sigma\in O^+(F_\frak p M_\frak p): \ x_\frak p\in \sigma (M_\frak p) \} $$ where $x_\frak p\in M_\frak p$ with $Q(x_\frak p)=2$ for all $\frak p\in \Omega_F$.
 Since $$\theta_\frak p (X(M_\frak p, 2))=N_{\frak P\mid \frak p} (G_\frak P^\times) $$ by \cite[Satz 3(a)]{schulzepillot_Darstellung_1980} and \cite[Theorem 2.0]{Xu4}, one obtains that $2$ is a spinor exception for $\gen(M)$ by \cite[Theorem 4.1]{hsia_indefinite_1998}. In particular, $h(M)\geq 2$. Therefore $h(M)=2$ and the equality of (\ref{cls}) holds. Moreover $G$ is also the spinor class field of $\gen(M)$ in the sense of \cite[p.131]{hsia_indefinite_1998}.

 Since $3$ splits completely in $F$ and is inert in $\Bbb Q(\sqrt{2})/\Bbb Q$, the prime $\frak a=(3, 1-\sqrt{-14})$ of $F$ above $3$ is inert in $G/F$.
Define
$$L= (\frak a^{-1} x+ \frak a y) \perp \mathcal O_F z  .$$
Since
$$ \tau_{\pi_\frak a^{-1} x + y}(\pi_\frak a^{-1} x)= \pi_\frak a^{-1} x - \frac{2 \pi_\frak a^{-1}}{Q(y)+2\pi_\frak a^{-1}}(\pi_\frak a^{-1} x + y)= \frac{Q(y)}{2+Q(y) \pi_\frak a} x - \frac{2}{2+Q(y) \pi_\frak a}y \in M_\frak a $$
and
$$ \tau_{\pi_\frak a^{-1}x+y}(\pi_\frak a y) =\pi_\frak a y -\frac{2(\pi_\frak a Q(y)+1)}{Q(y)+2\pi_\frak a^{-1}} (\pi_\frak a^{-1} x+y)=- \frac{2(\pi_\frak a Q(y)+1)}{ 2+Q(y) \pi_\frak a} x-\frac{Q(y)\pi_\frak a^2}{ 2+Q(y) \pi_\frak a} y \in M_\frak a ,$$ this implies that $ \tau_{\pi_\frak a^{-1}x+y}(L_\frak a)=M_\frak a$ and $L\in \gen(M)$. Moreover, since the intersection ideal of $M$ and $L$ (see \cite[p.134]{hsia_indefinite_1998}) is $\frak a$ which is inert in $G/F$, one concludes that $L$ and $M$ are exactly representative elements of the two proper classes in $\gen(M)$ by \cite[Theorem 3.1]{hsia_indefinite_1998} and \cite[82:4]{omeara_quadratic_1963}.

The quadratic form corresponding to $M$ being $2xy+\sqrt{-14}y^2+z^2$, the identity
 \[
 2\cdot 3\cdot \sqrt{-14}+\sqrt{-14}\cdot (\sqrt{-14})^2+(4+\sqrt{-14})^2=2
 \]
 shows that $M$ represents $2$ over $\cO_F$. Since $2$ is a spinor exception for $\gen(M)$, one concludes that $L$ cannot represent $2$.

Since $\frak a^{-1}=(\frac{1+\sqrt{-14}}{3}, 3)$,  one can make the following substitution
$$ \begin{cases} u= \frac{1}{3} (1+ \sqrt{-14}) x+ 3 y \\
v = 2 x + (1-\sqrt{-14}) y \end{cases} $$
with $u, v\in \frak a x+ \frak a^{-1} y$. Then $\frak a x+ \frak a^{-1} y= \cO_F u + \cO_F v$ by \cite[81:8]{omeara_quadratic_1963}. The corresponding quadratic form of $L$ is
$$(2+11 \sqrt{-14})x^2+(106+6\sqrt{-14}) xy + (32-17\sqrt{-14}) y^2 + z^2 .  $$

 Let $L_m= (\frak a^{-1} x+ \frak a y) \perp \mathcal O_F m z \subseteq L $  with the corresponding quadratic form $$ (2+11 \sqrt{-14})x^2+(106+6\sqrt{-14}) xy + (32-17\sqrt{-14}) y^2 + m^2 z^2 . $$ Then $L_m$ cannot represent $2$ either. On the other hand,  $(L_m)_\frak p$ is 1-universal over $\cO_\frak p$ for all $\frak p\in \Omega_F\setminus \infty_F$ by \cite[Proposition 2.3 and Proposition 2.17]{xu_indefinite_2020} as desired.
\end{proof}

\

\noindent \emph{Acknowledgments}. Part of the paper comes from discussions during our visit to Xi'an Jiaotong University in July 2021. We would like to thank Prof.\;Ping Xi for the invitation and kind hospitality.  Zilong He and Yong Hu are supported by the National Natural Science Foundation of China (grant no.\,12171223). Fei Xu is supported by the National Natural Science Foundation of China (grant no.\,11631009).



\

\

Contact information of the authors:

\

Zilong HE

Department of Mathematics

Southern University of Science and Technology

Shenzhen 518055, China

Email: hezl6@sustech.edu.cn

\

Yong HU

Department of Mathematics

Southern University of Science and Technology


Shenzhen 518055, China

Email: huy@sustech.edu.cn

\

Fei XU

School of Mathematical Science

Capital Normal University

Beijing 100048, China

Email: 6091@cnu.edu.cn

\

\end{document}